\documentclass[reqno,10pt]{amsart}

\usepackage{amssymb, latexsym}
\usepackage{hyperref}
\usepackage{amsmath}
\usepackage{enumerate}
\usepackage{amsfonts}
\usepackage{graphicx}
\usepackage{mathrsfs}

\newcommand{\eps}{{\varepsilon}}

\theoremstyle{plain}
\newtheorem{theorem}{Theorem}
\newtheorem{proposition}[theorem]{Proposition}
\newtheorem{lemma}[theorem]{Lemma}
\newtheorem{corollary}[theorem]{Corollary}

\theoremstyle{definition}
\newtheorem{definition}[theorem]{Definition}
\newtheorem{remark}[theorem]{Remark}

\newcommand{\qtq}[1]{\quad\text{#1}\quad}

\numberwithin{equation}{section}
\numberwithin{theorem}{section}

\numberwithin{equation}{section}

\newcommand{\Z}{{\mathbb{Z}}}
\newcommand{\C}{{\mathbb{C}}}
\newcommand{\R}{{\mathbb{R}}}

\newcommand{\E}{\mathcal{E}}

\newcommand{\K}{\mathcal{K}}

\let\Re=\undefined\DeclareMathOperator*{\Re}{Re}
\let\Im=\undefined\DeclareMathOperator*{\Im}{Im}

\renewcommand{\L}{\mathcal{L}_a}

\begin{document}

\title[NLS with inverse-square potential]{Scattering in $H^1$ for the intercritical NLS with an inverse-square potential}
\author[J. Lu]{Jing Lu}
\address{School of Mathematical Sciences, Beijing Normal University,\ Beijing,\ China,\
100875} \email{lujingaaaa@126.com}

\author[C. Miao]{Changxing Miao}
\address{Institute of Applied Physics and Computational Mathematics, Beijing, China, 100088}
\email{ miao\_changxing@iapcm.ac.cn}

\author[J. Murphy]{Jason Murphy}
\address{Department of Mathematics,
University of California, Berkeley, USA}
\email{murphy@math.berkeley.edu}

\begin{abstract}  We study the nonlinear Schr\"odinger equation with an inverse-square potential in dimensions $3\leq d \leq 6$.  We consider both focusing and defocusing nonlinearities in the mass-supercritical and energy-subcritical regime.  In the focusing case, we prove a scattering/blowup dichotomy below the ground state.  In the defocusing case, we prove scattering in $H^1$ for arbitrary data. \end{abstract}

\maketitle

\section{Introduction}

We consider the Cauchy problem for nonlinear Schr\"odinger equations (NLS) with an inverse-square potential:
\begin{align*}
\begin{cases}
  \label{nls}\tag{$\text{NLS}_a$}
(i\partial_t - \L)u=\mu|u|^{\alpha}u, \quad (t,x)\in\R\times\R^d, \\
u(0, x)=u_0(x),
\end{cases}
\end{align*}
in dimensions $d\geq 3$.  Here we consider an inverse square potential, i.e.  
\begin{equation}\label{LA}
\L=-\Delta + \tfrac{a}{|x|^2}\qtq{for} a>-\bigl(\tfrac{d-2}{2}\bigr)^2.
\end{equation}
More precisely, we consider the Friedrichs extension of the quadratic form	$Q$ defined on $C_c^\infty(\R^d\backslash\{0\})$ via
\[
Q(f) = \int_{\R^d} |\nabla f(x)|^2 + \tfrac{a}{|x|^2}|f(x)|^2\,dx. 
\]
The choice of the Friedrichs extension is natural from a physical point of view; furthermore, when $a=0$, $\L$ reduces to the standard Laplacian $-\Delta$. For more details, see for example \cite{KSWW}. 

We choose the power $\alpha$ in \eqref{nls} to be \emph{intercritical}, i.e. mass-supercritical but energy-subcritical (cf. the discussion below): 
\begin{equation}\label{intercritical}
\tfrac4d<\alpha<\tfrac4{d-2}.
\end{equation}
We consider $\mu\in\{\pm1\}$, where $\mu=1$ gives the defocusing case and $\mu=-1$ gives the focusing case.

The restriction on $a$ in \eqref{LA} guarantees positivity of $\L$.  In fact, by the sharp Hardy inequality,
\begin{equation}\label{q-equiv}
Q(f) = \| \sqrt{\L}f\|_{L_x^2}^2 \sim \|\nabla f\|_{L_x^2}^2\qtq{for} a>-\bigl(\tfrac{d-2}{2}\bigr)^2.
\end{equation}
In particular, the Sobolev space $\dot H_x^1$ is isomorphic to the space $\dot H_a^1$ defined in terms of $\L$.  The equivalence of other Sobolev spaces plays an important role in the well-posedness theory for \eqref{nls}; see Section~\ref{S:LWP} below. 

Solutions to \eqref{nls} conserve the \emph{mass} and \emph{energy}, defined respectively by
\begin{align*}
& M(u(t)) := \int_{\R^d} |u(t,x)|^2 \,dx, \\
& E_a(u(t)) := \int_{\R^d} \tfrac12|\nabla u(t,x)|^2 + \tfrac{a}{2|x|^2} |u(t,x)|^2 + \tfrac{\mu}{\alpha+2} |u(t,x)|^{\alpha+2} \,dx.
\end{align*}

When $a=0$, \eqref{nls} reduces to the `free' NLS:
\begin{equation}\label{nls0}\tag{$\text{NLS}_0$}
(i\partial_t+\Delta) u = \mu |u|^{\alpha}u.
\end{equation}
Like \eqref{nls0}, the equation \eqref{nls} enjoys the scaling symmetry
\begin{equation}\label{scaling}
u(t,x) \mapsto u^\lambda(t,x) : = \lambda^{\frac{2}{\alpha}} u(\lambda^2t, \lambda x).
\end{equation}
This symmetry identifies $\dot H_x^{s_{c}}(\R^d)$ as the scaling-critical space of initial data, where $s_{c}=\frac{d}{2}-\frac{2}{\alpha}$.  

The \emph{mass-critical} problem corresponds to $s_c=0$ (or $\alpha=\frac4d$), in which case $M(u)\equiv M(u^\lambda)$.  The \emph{energy-critical} problem corresponds to $s_c=1$ (or $\alpha=\frac4{d-2}$), in which case $E_a(u)\equiv E_a(u^\lambda)$.  The condition \eqref{intercritical} is equivalent to $s_c\in(0,1)$, which we call the \emph{intercritical} case. 

In contrast to \eqref{nls0}, the equation \eqref{nls} with $a\neq 0$ is \emph{not} space-translation invariant.  The presence of a broken symmetry in \eqref{nls} plays an important role in the analysis throughout the paper. 

In this paper, we study global well-posedness and scattering for \eqref{nls} for initial data $u_0\in H^1$.  Such data have finite mass and energy; indeed, this follows from \eqref{q-equiv} and the following Gagliardo--Nirenberg inequality:
\begin{equation}\label{E:GN}
\|f\|_{L_x^{\alpha+2}}^{\alpha+2} \leq C_a \|f\|^{\frac{4-(d-2)\alpha}{2}}_{L_x^2}\|f\|_{\dot H_a^1}^{\frac{d\alpha}{2}},
\end{equation}
where $C_a$ denotes the sharp constant in the inequality above.  Note that $C_a$ is finite in light of the standard Gagliardo--Nirenberg inequality and \eqref{q-equiv}.  The inequality \eqref{E:GN} plays a key role throughout the paper; it is discussed further in Section~\ref{S:var}.

Before stating our results, we briefly discuss the relevant past results on \eqref{nls0} and \eqref{nls} in the intercritical setting. 

\subsection{Discussion of past results} For the defocusing intercritical free NLS, one has scattering in $H^1$ \cite{GV, Na}, that is, for any $u_0\in H_x^1$ there exist a global solution $u$ and unique $u_\pm\in H^1$ such that
\[
\lim_{t\to\pm\infty} \| u(t) - e^{it\Delta}u_{\pm}\|_{H_x^1(\R^d)} = 0.
\]
Here $e^{it\Delta}$ denotes the Schr\"odinger group.  For \eqref{nls} in the defocusing intercritical setting, the authors of \cite{ZhaZhe} proved scattering in $H^1$ in the regime
\begin{equation}\label{zz}
\begin{cases} a\geq 0 & d = 3, \\ a>-\bigl(\tfrac{d-2}{2}\bigr)^2 + \bigl(\tfrac{2}{\alpha+2}\bigr)^2 & d\geq 4.\end{cases}
\end{equation}
That is, they showed that there exist unique $u_\pm\in H^1$ so that
\[
\lim_{t\to\pm\infty} \|u(t) - e^{-it\L}u_{\pm} \|_{H_x^1(\R^d)} = 0.
\]
The restrictions in \eqref{zz} stemmed from the interaction Morawetz inequality. 

For the \emph{focusing} intercritical free NLS, there exists a global nonscattering solution, namely, the ground state soliton $u(t) = e^{it}Q_0$, where $Q_0$ is the unique, positive, radial, decaying solution to 
\[
\Delta Q_0 - Q_0 + Q_0^{\alpha+1} = 0.
\]
In \cite{cazenave, DHR, Guevara, HR}, a blowup/scattering dichotomy was established `below the ground state'.  In particular, \cite{DHR, HR} considered the cubic NLS in three dimensions (see also \cite{DM}), while \cite{cazenave, Guevara} considered the full intercritical range.  To make this precise, one can define the following quantities (for some fixed $0<s_c<1$):
\[
\E_0 = M(Q_0)^\sigma E_0(Q_0), \quad \K_0 = \|Q\|_{L_x^2}^\sigma \|Q\|_{\dot H_x^1}, \quad \sigma = \tfrac{1}{s_c}-1.
\]
Then one has the following:
\begin{theorem}[Scattering/blowup dichotomy, free case \cite{cazenave, DHR, Guevara, HR}]\label{T:DHR1} Let $\mu=-1$ and let $\alpha$ satisfy \eqref{intercritical}.  Suppose $u_0\in H^1$ satisfies $M(u_0)^\sigma E_0(u_0)<\E_0$ and let $u$ be the corresponding solution to \eqref{nls0} with initial data $u_0$.

If $\|u_0\|^\sigma_{L^2}\|u_0\|_{\dot H^1}>\K_0$ and $u_0$ is radial or $xu_0\in L^2$, then $u$ blows up in finite time in both time directions.

If $\|u_0\|^\sigma_{L^2}\|u_0\|_{\dot H^1}<\K_0$, then $u$ is global and scatters.

Furthermore, if $\psi\in H^1$ satisfies $\frac{1}{2} \|\psi\|^{2\sigma}_{L^2}\| \psi\|^{2}_{\dot H^1}< \E_0,$ then there exists a global solution to \eqref{nls0} that scatters to $\psi$ forward in time.  The analogous statement holds backward in time.\end{theorem}

For \eqref{nls} in the focusing intercritical regime, an analogous result was established in \cite{KVMZ}.  In particular, \cite{KVMZ} considered a cubic nonlinearity in three space dimensions.  In this case, when $a\in(-\frac14,0]$ one can construct a solution to the elliptic problem 
\begin{equation}\label{ell}
-\L Q_a - Q_a + |Q_a|^2Q_a = 0,
\end{equation}
as an optimizer to the Gagliardo--Nirenberg inequality \eqref{E:GN}; when $a>0$, no optimizer exists.  In this case, one defines the quantities
\[
\E_a = M(Q_{a\wedge 0})E_{a\wedge 0}(Q_{a\wedge 0}),\quad 
\K_a = \|Q_{a\wedge 0}\|_{L_x^2} \|Q_{a\wedge 0}\|_{\dot H_{a\wedge 0}^1},
\]
where $a\wedge 0=\min\{a,0\}$.  
The main result in \cite{KVMZ} is the following: 

\begin{theorem}[Scattering/blowup dichotomy \cite{KVMZ}]\label{T:main1} Let $\mu=-1$, $a>-\frac14$, $\alpha=2$, and $d=3$. Suppose $u_0\in H^1(\R^3)$ satisfies $M(u_0)E_a(u_0)<\E_a$ and let $u$ be the corresponding solution to \eqref{nls} with initial data $u_0$.

If $\|u_0\|_{L^2}\|u_0\|_{\dot H_a^1}>\K_a$ and $u_0$ is radial or $xu_0\in L^2$, then $u$ blows up in finite time in both time directions.

If $\|u_0\|_{L^2}\|u_0\|_{\dot H_a^1}<\K_a$, then $u$ is global and scatters.
\end{theorem}

\subsection{Discussion of main results} In this paper, we firstly address the scattering theory for \eqref{nls} in the defocusing intercritical setting, extending the results of \cite{ZhaZhe}. For the focusing problem, we adapt the arguments of \cite{KVMZ} to prove a scattering/blowup dichotomy below the ground state for the full intercritical regime, giving a result analogous to that of \cite{cazenave} for the free NLS.

Our results require a local well-posedness theory in $H^1$ for \eqref{nls}.  This leads to restrictions on the range of $(d,a)$ that we can consider, as we now briefly explain.  As in the case of \eqref{nls0}, Strichartz estimates play a key role in the local theory.  For the case of \eqref{nls}, the full range of Strichartz estimates were established in \cite{BPSTZ}.  We also need to estimate powers of $\L$ applied to the nonlinearity.  To get the requisite fractional calculus estimates for $\L$, we rely on the equivalence of Sobolev spaces (proved in \cite{KMVZZ1}) to exchange powers of $\L$ and powers of $-\Delta$ (for which fractional calculus estimates are known).  This approach leads to a restriction on the range of $(d,a)$ that we can treat. Specifically, we consider the following ranges:
\begin{align}\label{important}
\begin{cases}
a>-\bigl(\tfrac{d-2}{2}\bigr)^2 & \text{if}\quad d=3\qtq{and} \tfrac43<\alpha\leq 2, \\
a>-\bigl(\tfrac{d-2}{2}\bigr)^2+\big(\tfrac{d-2}{2}-\tfrac{1}{\alpha}\big)^2& \text{if}\quad 3\leq d \leq 6 \qtq{and}\tfrac{2}{d-2}\vee \tfrac{4}{d}< \alpha<\tfrac{4}{d-2}.\end{cases}
\end{align}

Here $a\vee b := \max\{a,b\}$.  As will be discussed in Section~\ref{S:LWP}, one can prove a local theory in critical spaces for a larger range of $(d,a)$; however, we need to rely on conservation laws and hence we work at the level of $H^1$.  For the specific estimates leading to the restrictions \eqref{important}, see \eqref{why}.  

 Our first result is for the defocusing case.

\begin{theorem}[Scattering]\label{T:main2} Assume $\mu=1$ and that $(d,a,\alpha)$ satisfy \eqref{important}.  Then for any $u_0\in H^1$, the solution to \eqref{nls} with initial data $u_0$ is global and scatters.
\end{theorem}

In the focusing case, we prove a result analogous to Theorem~\ref{T:DHR1} and Theorem~\ref{T:main1}, namely, a scattering/blowup dichotomy below the ground state.  In particular, in Section~\ref{S:var} we will see that there exist optimizers $Q_a$ to the Gagliardo--Nirenberg inequality \eqref{E:GN} for $a\leq 0$, which solve the elliptic equation 
\begin{equation}\label{ell}
-\L Q_a - Q_a + |Q_a|^{\alpha}Q_a = 0.
\end{equation}
For $a>0$, $C_a=C_0$ but no optimizers exist.  As above, we define the thresholds
\begin{equation}\label{thresholds-Q}
\E_a := M(Q_{a\wedge 0})^{\sigma}E_{a\wedge 0}(Q_{a\wedge 0}),\quad \K_a := \|Q_{a\wedge 0}\|^{\sigma}_{L_x^2}\|Q_{a\wedge 0}\|_{\dot H_{a\wedge 0}^1},
\end{equation}
where $\sigma := \tfrac{1}{s_c}-1$ and $a\wedge b := \min\{a,b\}$. We remark that $\E_a$ and $\K_a$ may be described purely in terms of the sharp constant $C_a$ (see Section~\ref{S:var}).  Our second result is the following:

\begin{theorem}[Scattering/blowup dichotomy]\label{T:main} Assume $\mu=-1$ and that $(d,a,\alpha)$ satisfy \eqref{important}. Suppose that $u_0\in H^1$ satisfies $M(u_0)^{\sigma}E_a(u_0)<\E_a$ and let $u$ be the corresponding solution to \eqref{nls} with initial data $u_0$. 
\begin{itemize}
\item[(i)] If $\|u_0\|_{L^2}\|u_0\|_{\dot H_a^1}>\K_a$ and $u_0$ is radial or $xu_0\in L^2$, then $u$ blows up in finite time in both time directions.
\item[(ii)] If $\|u_0\|_{L^2}\|u_0\|_{\dot H_a^1}<\K_a$, then $u$ is global and scatters.
\end{itemize}
\end{theorem}

Our arguments parallel those of \cite{KVMZ}, which treated the cubic problem in three dimensions.  New technical obstructions appear throughout the arguments, related especially to the problem of equivalence of Sobolev spaces. The blowup result in Theorem~\ref{T:main} will follow from fairly standard virial arguments; thus, we focus on discussing the scattering results in Theorem~\ref{T:main2} and Theorem~\ref{T:main}.  

For the scattering results, we adopt the concentration compactness approach to induction on energy:  We first show that if the scattering result is false, then we may find a minimal blowup solution that is global-in-time and has a precompact orbit in $H^1$. We then use a localized virial argument to rule out the existence of such solutions; in the focusing case, the sub-threshold assumption guarantees the requisite coercivity of the virial estimate.

The broken translation symmetry present in \eqref{nls} plays an important role in the analysis, particularly in the construction of minimal blowup solutions.  This construction relies on linear and nonlinear profile decompositions.  The most delicate point comes in the construction of scattering solutions corresponding to nonlinear profiles with translation parameters tending to spatial infinity (cf. Theorem~\ref{T:embedding}).  For such profiles, we rely on the scattering results for the free NLS (e.g. Theorem~\ref{T:main1} above) to construct a scattering solution to \eqref{nls0}.  We then show that this solution approximately solves \eqref{nls} and invoke a stability result to deduce the existence of a true scattering solution to \eqref{nls}.  As a consequence of these arguments, we find that minimal blowup solutions are pre-compact in $H^1$ without modding out by a spatial center; this facilitates a direct implementation of the standard localized virial arguments. 

\subsection{Outline of the paper} The rest of the paper is organized as follows. In Section~\ref{S:prelim}, we first introduce notation. We also discuss harmonic analysis tools adapted to $\L$, as well as the local theory for \eqref{nls}.  We finally discuss virial identities and the variational analysis related to the sharp Gagliardo--Nirenberg inequality.  In Section~\ref{S:cc}, we develop the requisite concentration compactness tools adapted to the Strichartz estimate for $\L$.  We also prove the embedding result nonlinear profiles, Theorem~\ref{T:embedding}. In Section~\ref{S:exist}, we show that if the scattering results fail, then there exist minimal blowup solutions.  In Section~\ref{S:not-exist}, we preclude the possibility of such minimal blowup solutions, completing the proofs of Theorem~\ref{T:main2} and Theorem~\ref{T:main}(ii).  Finally, in Section~\ref{S:blowup}, we prove the blowup result, Theorem~\ref{T:main}(i). 

\subsection*{Acknowledgements} J.M. was supported by the NSF Postdoctoral Fellowship DMS-1400706. C.M. was partly supported by the NSF of China (No. 11671047)

\section{Preliminaries}\label{S:prelim}
The notation $A\lesssim B$ means that ${A}\leqslant{CB}$ for some constant $C>0$. If ${A}\lesssim{B}\lesssim{A}$,
 we write ${A}\sim{B}$. We write $A\wedge B = \min\{A,B\}$,  $A\vee B =\max\{A, B\}$,  and $\langle x\rangle =\sqrt{1+|x|^2}.$  We use $L^{q}_{t}L^{r}_{x}$ space-time norms defined via
\[
\|f\|_{L^{q}_{t}L^{r}_{x}(I\times\R^d)}:=\Big(\int_{I}\|f(t)\|^{q}_{L^r(\R^d)}dt\Big)^{\frac{1}{q}}
\]
for any space-time slab $I\times{\mathbb{R}^d}$. We make the usual modifications when $q$ or $r$ equals $\infty$. When $q=r$, we abbreviate $L^{q}_{t}L^{r}_{x}$ by $L^{q}_{t,x}$.  To shorten formulas, we often omit $\R^d$ or $I\times\R^d$. For $r\in[1,\infty]$ we let $r'\in[1,\infty]$ denote the H\"older dual, i.e. the solution to $\tfrac{1}{r}+\tfrac{1}{r'}=1$.

We write $x+$ to denote $x+\eps$ for some small $\eps>0$, and similarly for $x-$. 

We define Sobolev spaces in terms of $\L$ via
\[
\|f\|_{\dot H^{s,r}_a(\R^d)}= \|(\mathcal{L}_a)^{\frac{s}2} f\|_{L_x^r(\R^d)} \qtq{and} \|f\|_{H^{s,r}_a(\R^d)}= \|(1+ \mathcal{L}_a)^\frac{s}2 f\|_{L_x^r(\R^d)}.
\]
We abbreviate $\dot H^{s}_a(\R^d)=\dot H^{s,2}_a(\R^d)$ and $H^{s}_a(\R^d)=H^{s,2}_a(\R^d)$. 

\subsection{Harmonic analysis adapted to $\L$} In this section, we describe some harmonic analysis tools adapted to the operator $\L$.  The primary reference for this section is \cite{KMVZZ1}.

Recall that by the sharp Hardy inequality, one has
\begin{equation}\label{iso}
\|\sqrt{\L}\, f\|_{L_x^2}^2 \sim \|\nabla f\|_{L_x^2}^2\qtq{for}  a>-(\tfrac{d-2}2)^{2}.
\end{equation}
Thus, the operator $\L$ is positive for $a> -(\frac{d-2}2)^2$.  To state the estimates below, it is useful to introduce the parameter
\begin{equation}\label{rho}
\rho:=\tfrac{d-2}2-\bigr[\bigl(\tfrac{d-2}2\bigr)^2+a\bigr]^{\frac12}.
\end{equation}

We first give the estimates on the heat kernel associated to the operator $\mathcal{L}_a$.

\begin{lemma}[Heat kernel bounds, \cite{MS,LS}] \label{L:kernel}  Let $d\geq 3$ and $a> -(\tfrac{d-2}{2})^2$. There exist positive constants $C_1,C_2$ and $c_1,c_2$ such that for any $t>0$ and any $x,y\in\R^d\backslash\{0\}$,
\[
C_1(1\vee\tfrac{\sqrt{t}}{|x|})^\rho(1\vee\tfrac{\sqrt{t}}{|y|})^\rho t^{-\frac{d}{2}} e^{-\frac{|x-y|^2}{c_1t}} \leq e^{-t\L}(x,y) \leq
C_2(1\vee\tfrac{\sqrt{t}}{|x|})^\rho(1\vee\tfrac{\sqrt{t}}{|y|})^\rho t^{-\frac{d}{2}} e^{-\frac{|x-y|^2}{c_2t}}.
\]
\end{lemma}

The following result concerning equivalence of Sobolev spaces was established in \cite{KMVZZ1}; it plays an important role throughout this paper.

\begin{lemma}[Equivalence of Sobolev spaces, \cite{KMVZZ1}]\label{pro:equivsobolev} Let $d\geq 3$, $a> -(\frac{d-2}{2})^2$, and $0<s<2$. If $1<p<\infty$ satisfies $\frac{s+\rho}{d}<\frac{1}{p}< \min\{1,\frac{d-\rho}{d}\}$, then
\[
\||\nabla|^s f \|_{L_x^p}\lesssim_{d,p,s} \|(\L)^{\frac{s}{2}} f\|_{L_x^p}\qtq{for all} f\in C_c^\infty(\R^d\backslash\{0\}).
\]
If $\max\{\frac{s}{d},\frac{\rho}{d}\}<\frac{1}{p}<\min\{1,\frac{d-\rho}{d}\}$, then
\[
\|(\L)^{\frac{s}{2}} f\|_{L_x^p}\lesssim_{d,p,s} \||\nabla|^s f\|_{L_x^p} \qtq{for all} f\in C_c^\infty(\R^d\backslash\{0\}).
\]
\end{lemma}

Next, we recall some fractional calculus estimates due to Christ and Weinstein \cite{CW}.  Combining these estimates with Lemma~\ref{pro:equivsobolev}, we can deduce analogous statements for powers of $\L$ (with suitably restricted sets of exponents). 
\begin{lemma}[Fractional calculus]\text{ }
\begin{itemize}
\item[(i)] Let $s\geq 0$ and $1<r,r_j,q_j<\infty$ satisfy $\tfrac{1}{r}=\tfrac{1}{r_j}+\tfrac{1}{q_j}$ for $j=1,2$. Then
\[
\| |\nabla|^s(fg) \|_{L_x^r} \lesssim \|f\|_{L_x^{r_1}} \||\nabla|^s g\|_{L_x^{q_1}} + \| |\nabla|^s f\|_{L_x^{r_2}} \| g\|_{L_x^{q_2}}.
\]
\item[(ii)] Let $G\in C^1(\C)$ and $s\in (0,1]$, and let $1<r_1\leq \infty$  and $1<r,r_2<\infty$ satisfy $\tfrac{1}{r}=\tfrac{1}{r_1}+\tfrac{1}{r_2}$. Then
\[
\| |\nabla|^s G(u)\|_{L_x^r} \lesssim \|G'(u)\|_{L_x^{r_1}} \|u\|_{L_x^{r_2}}.
\]
\end{itemize}
\end{lemma}

We make use of Littlewood--Paley projections defined via the heat kernel:
\begin{align*}
P_N^a:=e^{-\mathcal L_a/N^2}-e^{-4\mathcal L_a/N^2}\qtq{for} N \in 2^{\mathbb{Z}}.
\end{align*}
In order to state the following results, it is convenient to define
\[
\tilde q := \begin{cases} \infty & \text{if }a\geq 0, \\ \tfrac{d}{\rho} & \text{if }-(\frac{d-2}{2})^2< a < 0. \end{cases}
\]
We write $\tilde q'$ for the dual exponent to $\tilde q$.  

We begin with several lemmas from \cite{KMVZZ1}, which were proved using a Mihklin-type multiplier theorem for functions of $\L$.

\begin{lemma}[Expansion of the identity \cite{KMVZZ1}] Let $\tilde q' < r < \tilde q$. Then
\[
f= \sum_{N\in 2^{\mathbb{Z}}} P_N^a f\qtq{as elements of}  L_x^r.
\]
\end{lemma}

\begin{lemma}[Bernstein estimates \cite{KMVZZ1}]\label{L:Bernie} Let $\tilde q'<q\leq r<\tilde q$.  Then
\begin{itemize}
\item[(i)] The operators $P^a_{N}$ are bounded on $L_x^r$.
\item[(ii)] The operators $P^a_{N}$ map $L_x^q$ to $L_x^r$, with norm $O(N^{\frac dq-\frac dr})$.
\item[(iii)] For any $s\in \R$,
\[
N^s\|P^a_N f\|_{L_x^r}  \sim \bigl\|\L^{\frac s2}P^a_N f\bigr\|_{L_x^r}.
\]
\end{itemize}
\end{lemma}

\begin{lemma}[Square function estimate \cite{KMVZZ1}]\label{T:sq}
Let $0\leq s<2$ and $\tilde q'<r<\tilde q$. Then
\begin{align*}
\biggl\|\biggl(\sum_{N\in2^\mathbb{Z}} N^{2s}| P^a_N f|^2\biggr)^{\!\!\frac 12}\biggr\|_{L_x^r} \sim \|(\L)^{\frac s2}f\|_{L_x^r}.
\end{align*}
\end{lemma}

We also record a refined Fatou lemma for use in Section~\ref{S:cc}.

\begin{lemma}[Refined Fatou \cite{BL}]\label{L:RF} Let $1\leq r<\infty$ and let $\{f_n\}$ be a bounded sequence in $L_x^r$. If $f_n\to f$ almost everywhere, then
\[
\int \bigl| |f_n|^r - |f_n-f|^r - |f|^r \bigr| \,dx \to 0.
\]
\end{lemma}

Strichartz estimates for the propagator $e^{-it\L}$ were proved in \cite{BPSTZ}.  Combining these with the Christ--Kiselev lemma \cite{CK}, we arrive at the following:

\begin{proposition}[Strichartz \cite{BPSTZ}] Fix $a>-(\tfrac{d-2}{2})^2$. The solution $u$ to 
\[
(i\partial_t-\L)u = F
\]
on an interval $I\ni t_0$ obeys
\[
\|u\|_{L_t^q L_x^r(I\times\R^d)} \lesssim \|u(t_0)\|_{L_x^2(\R^d)} + \|F\|_{L_t^{\tilde q'} L_x^{\tilde r'}(I\times\R^d)}
\]
for any $2\leq q,\tilde q\leq\infty$ with $\frac{2}{q}+\frac{d}{r}=\frac{2}{\tilde q}+\frac{d}{\tilde r}= \frac{d}2$ and $(q,\tilde q)\neq (2,2)$. \end{proposition}
We call such pairs $(q,r)$ and $(\tilde{q},\tilde{r})$ \emph{admissible} pairs.

\subsection{Function spaces}\label{S:FS}

We need to take some care to work in function spaces for which we have equivalence of Sobolev spaces (cf. Lemma~\ref{pro:equivsobolev}).  Many of the exponents we use are complicated combinations of $\alpha$ and $d$.  For this reason, we introduce some notation for frequently-used exponents and function spaces. 

First, it will be convenient to use the following notation:
\[
S^s_a(I) = L_t^2 H_a^{s,\frac{2d}{d-2}}\cap L_t^\infty H_a^s(I\times\R^d) \qtq{and}
\dot S^s_a(I) = L_t^2 \dot H_a^{s,\frac{2d}{d-2}}\cap L_t^\infty \dot H_a^s(I\times\R^d).
\]

We let
\[
\begin{aligned}
&q_0 = \tfrac{\alpha(d+2)}{2}, \quad r_0 = \tfrac{2d\alpha(d+2)}{d\alpha(d+2)-8},  
\end{aligned}
\]
Then $(q_0,r_0)$ is an admissible pair.  The  $L_{t,x}^{q_0}$-norm is critical for \eqref{nls} (i.e. invariant under \eqref{scaling}) and will be used to give a scattering criterion below.  By Sobolev embedding, one has $\dot H_x^{s_c,r_0} \hookrightarrow L_{x}^{q_0}.$  Furthermore, for $(d,a,\alpha)$ satisfying \eqref{important}, we have by Lemma~\ref{pro:equivsobolev} that $\dot H_a^{s_c,r_0}$ and $\dot H_x^{s_c,r_0}$ are equivalent. 

In conjunction with the spaces introduced above, we will often use the particular dual admissible pair
\[
(\rho,\gamma) = (\tfrac{\alpha(d+2)}{2(\alpha+1)}, \tfrac{2\alpha d(d+2)}{\alpha d(d+6)-8}). 
\]

Further specific exponents to be used throughout the paper will be introduced in Remark~\ref{remark1}. 

\subsection{Convergence of operators}\label{S:coo}  In this section we recall some results from \cite{KMVZZ2} concerning the convergence of certain linear operators arising from the lack of translation symmetry for $\L$.  These will be useful in Sections~\ref{S:cc} and~\ref{S:exist}.

\begin{definition}\label{D:ops} Suppose $\{x_n\}\subset\R^d$. We define
\[
\L^n = -\Delta + \tfrac{a}{|x+x_n|^2} \qtq{and} 
\L^\infty = \begin{cases} -\Delta + \tfrac{a}{|x+x_\infty|^2}, & \text{if}\quad x_n\to x_\infty\in \R^d, \\
    -\Delta, & \text{if}\quad |x_n|\to\infty.\end{cases}
\]
In particular, $\L[\phi(x-x_n)] = [\L^n \phi](x-x_n).$
\end{definition}

 The operators $\L^\infty$ appear as limits of the operators $\L^n$, as in the following:

\begin{lemma}[Convergence of operators \cite{KMVZZ2}] \label{L:coo} Let $a>-(\tfrac{d-2}{2})^2$. Suppose $\tau_n\to \tau_\infty\in\R$ and $\{x_n\}\subset\R^d$ satisfies $x_n\to x_\infty\in\R^d$ or $|x_n|\to\infty$. Then,
\begin{align}
\label{coo1}
&\lim_{n\to\infty} \|\L^n \psi - \L^\infty \psi \|_{\dot H_x^{-1}} = 0\qtq{for all}  \psi\in \dot H_x^1, \\
\label{coo3}
&\lim_{n\to\infty} \|\bigl( e^{-i\tau_n\L^n}-e^{-i\tau_\infty\L^\infty}\bigr)\psi\|_{\dot H_x^{-1}} =0 \qtq{for all}  \psi\in \dot H_x^{-1}, \\
\label{coo4}
&\lim_{n\to\infty} \| \bigl[({\L^n})^{\frac12} -(\L^\infty)^{\frac12}\bigr]\psi\|_{L_x^2} = 0 \qtq{for all}  \psi\in\dot H_x^1.
\end{align}
Furthermore, for any $2< q\leq \infty$ and $\frac{2}{q}+\frac{d}{r}=\frac{d}{2}$,
\begin{equation}
\label{coo5}
\lim_{n\to\infty} \|\bigl(e^{-it\L^n}-e^{-it\L^\infty}\bigr)\psi\|_{L_t^q L_x^r(\R\times\R^d)} = 0 \qtq{for all}  \psi\in L_x^2.
\end{equation}
Finally, if $x_\infty\neq 0$, then for any $t>0$,
\begin{align}
\label{coo2}
\lim_{n\to\infty} \| [e^{-t\L^n}-e^{-t\L^\infty}]\delta_0 \|_{\dot H_x^{-1}} = 0.
\end{align}
\end{lemma}

In \cite[Corollary~3.4]{KMVZZ2}, the authors use \eqref{coo4} and \eqref{coo5} to prove
\[
\lim_{n\to\infty} \|e^{it_n\L^n}\psi\|_{L_x^{\frac{2d}{d-2}}}=0\qtq{for} \psi\in \dot H_x^1.
\]
Interpolating this with $L_x^2$-boundedness yields the following corollary.
\begin{corollary}\label{L4-to-zero} For $\{x_n\}\subset\R^d$, $t_n\to\pm\infty$, and $\psi \in H_x^1$, we have
\[
\lim_{n\to\infty} \|e^{it_n\L^n}\psi \|_{L_x^{\alpha+2}} = 0.
\]
\end{corollary}

We record one final corollary:
\begin{corollary}\label{S-to-zero} Let $a>-(\tfrac{d-2}{2})^2$. Suppose $x_n\to x_\infty\in\R^d$ or $|x_n|\to\infty$.  Then
\[
\lim_{n\to\infty} \|(e^{-it\L^n} - e^{-it\L^\infty})\psi\|_{L_{t,x}^{q_0}(\R\times\R^d)} = 0 \qtq{for all} \psi\in H_x^1.
\]
\end{corollary}

\begin{proof} Fix $\psi\in H^1_x$. By Sobolev embedding, it suffices to show
\[
\lim_{n\to\infty}\||\nabla|^{s_{c}}(e^{-it\L^n}-e^{-it\L^\infty})\psi\|_{L_t^{q_0} L_x^{r_0}}=0.
\]
To this end, we first use \eqref{coo5} to see that
\[
\lim_{n\to\infty}\| (e^{-it\L^n}-e^{-it\L^\infty})\psi\|_{L_t^{q_0}L_x^{r_0}} = 0.
\]
On the other hand, by equivalence of Sobolev spaces and Strichartz, we have
\[
\sup_n \||\nabla|^{s_c+}(e^{-it\L^n} -e^{-it\L^\infty})\psi\|_{L_t^{q_0}L_x^{r_0}}\lesssim \||\nabla|^{s_c+}\psi\|_{L_x^2}\lesssim 1.
\]
The result now follows by interpolation.\end{proof}

\subsection{Local well-posedness and stability}\label{S:LWP}
We next discuss the local theory for \eqref{nls}.  We need to consider both the subcritical and critical well-posedness results.  It is convenient to use subcritical results so that we can capitalize on \emph{a priori} $H^1$-bounds to deduce global existence; on the other hand, it is natural to address scattering via critical space-time bounds. 

We begin by making our notion of solution precise. 
\begin{definition}[Solution]\label{def:soln} Let $t_0\in\R$ and $u_0\in H_a^1(\R^d)$. Let $I$ be an interval containing $t_0$. We call $u:I\times\R^d\to\C$ a \emph{solution} to
\[
(i\partial_t - \L)u = \mu |u|^\alpha u,\quad u(t_0)=u_0
\]
if it belongs to $C_t H_a^1(K\times\R^d)\cap S_a^1(K)$ for any compact $K\subset I$ and obeys the Duhamel formula
\begin{equation}\label{duhamel}
u(t) = e^{-i(t-t_0)\L}u_0-i\mu\int_{t_0}^t e^{-i(t-s)\L}\bigl(|u|^\alpha u\bigr)(s)\,ds
\end{equation}
for all $t\in I$. We call $I$ the \emph{lifespan} of $u$. We call $u$ a \emph{maximal-lifespan solution} if it cannot be extended to a strictly larger interval. We call $u$ \emph{global} if $I=\R$. 
\end{definition}

\begin{theorem}[Local well-posedness]\label{T:LWP}
Let $t_0\in\R$ and $u_0\in H_x^{1}(\R^d)$.  Suppose $(d,a,\alpha)$ satisfy \eqref{important}. Then the following hold.
\begin{itemize}
\item[(i)] There exist $T=T(\|u_0\|_{H_a^{1}})>0$ and a unique solution $u:(t_0-T,t_0+T)\times\R^d\to\C$ with $u(t_0)=u_0$.  In particular, if $u$ remains uniformly bounded in $H_a^{1}$ throughout its lifespan, then $u$ extends to a global solution.
\item[(ii)]  There exists $\eta_0>0$ such that if
\[
\| e^{-i(t-t_0)\L}u_0\|_{L_{t,x}^{q_0}((t_0,\infty)\times\R^d)}<\eta\qtq{for some}0<\eta<\eta_0,
\]
then the solution to \eqref{nls} with $u(t_0)=u_0$ is forward-global and satisfies
\[
\| u\|_{L_{t,x}^{q_0}((t_0,\infty)\times\R^d)}\lesssim \eta.
\]
The analogous statement holds backward in time (and on all of $\R$). 
\item[(iii)] For any $\psi\in H_a^{1}$, there exist $T>0$ and a solution $u:(T,\infty)\times\R^d\to\C$ to \eqref{nls} such that 
\[
\lim_{t\to\infty}\|u(t)-e^{-it\L}\psi\|_{H_a^1}=0.
\]
The analogous statement holds backward in time. 
\end{itemize}

\end{theorem}

\begin{proof} By time-translation symmetry we may choose $t_0=0$.  The proofs follow along standard lines using the contraction mapping principle; in particular, for (i) and (ii) one constructs a solution satisfying the Duhamel formula \eqref{duhamel}, while for (iii) one needs to solve
\[
u(t) = e^{-it\L}\psi -i\mu\int_t^\infty e^{-i(t-s)\L}\bigl(|u(s)|^\alpha u(s)\bigr)\,ds. 
\]
We will show here the relevant nonlinear estimates.  For more details in a similar setting, see \cite{KVMZ}. 

For (i), we fix a space-time slab $(-T,T)\times\R^d$ and argue as follows.  We fix $\beta$ to be determined shortly and define the parameters
\[
(\tilde q,\tilde r) =(\tfrac{2\alpha\beta}{\beta-2}, \tfrac{2d\alpha\beta}{\beta(d\alpha-2)+4}),\quad s = \tfrac{d}{2}-\tfrac{2}{\alpha}(1-\tfrac{1}{\beta}). 
\]
We now choose
\[
2 \vee \tfrac{4}{4-\alpha(d-2)} < \beta < \begin{cases} \infty & \alpha\geq 1 \\ \frac{2}{1-\alpha} & \alpha < 1.\end{cases}
\]
The upper bound on $\beta$ guarantees that $(\tilde q,\tilde r)$ is an admissible pair.  The lower bound on $\beta$ guarantees that $s\leq 1$. The conditions on $a$ in \eqref{important} guarantee that $\dot H_x^{s,\tilde r}$ is equivalent to $\dot H_a^{s,\tilde r}$ (cf. Lemma~\ref{pro:equivsobolev}).  Note that to find $\beta$ adhering to the restrictions above requires that $(d-4)\alpha<2$; as we wish to consider the full intercritical range \eqref{intercritical}, we therefore restrict to dimensions $d\leq 6$.  (We could also include the range $\tfrac4d<\alpha<\tfrac2{d-4}$ in dimension $d=7$.) 

Having chosen parameters as above, we may now estimate by H\"older's inequality, Sobolev embedding, and the equivalence of Sobolev spaces:
\begin{equation}\label{why}
\begin{aligned}
\| |u|^\alpha u\|_{L_t^2 H_a^{1,\frac{2d}{d+2}}}  \lesssim \| |u|^\alpha u\|_{L_t^2 H_x^{1,\frac{2d}{d+2}}} & \lesssim T^{\frac1\beta} \|u\|_{L_t^{\tilde q} L_x^{d\alpha}}^\alpha \| u\|_{L_t^\infty H_x^1} \\
& \lesssim T^{\frac1\beta} \| |\nabla|^s u\|_{L_t^{\tilde q}L_x^{\tilde r}}^\alpha \|u\|_{L_t^\infty H_a^1} \\
& \lesssim T^{\frac1\beta} \| u\|_{L_t^{\tilde q}\dot H_a^{s,\tilde r}}^\alpha\|u\|_{L_t^\infty H_a^1}. 
\end{aligned}
\end{equation}
Using this estimate, one can close a contraction in the space $C_t H_a^1\cap L_t^{\tilde q} H_a^{1,\tilde r}$ on a sufficiently small time interval, where $T=T(\|u_0\|_{H_a^1}$). 

(ii) To prove a `critical' well-posedness result as in (ii), we would instead use the following nonlinear estimate:
\begin{equation}\label{critical-lwp}
\| |u|^\alpha u\|_{L_t^{\rho} H_a^{s_c,\gamma}}  \lesssim \|u\|_{L_{t,x}^{q_0}}^{\alpha} \| u\|_{L_t^{q_0} H_a^{s_c,r_0}}, 
\end{equation}
where once again we have relied on the equivalence of Sobolev spaces (and recall the notation from Section~\ref{S:FS}). Recalling that $\dot H_x^{s_c,r_0}\hookrightarrow L_x^{q_0}$, one can close a contraction in the space $C_t H_a^{s_c}\cap L_t^{q_0} H_a^{s_c,r_0}$, where we consider functions with small $L_{t,x}^{q_0}$-norm on this time interval.  To upgrade to a solution in the sense of Definition~\ref{def:soln}, we use Remark~\ref{remark1} below.

Using the same spaces as in (ii) and once again relying on Remark~\ref{remark1}, one can also prove item (iii).\end{proof}

\begin{remark}  If one is only interested in the critical well-posedness result, then the following conditions on $(d,a)$ are sufficient to get the necessary equivalence of Sobolev spaces:
\begin{align}\label{important2}
\begin{cases}
a> -\bigl(\tfrac{d-2}{2}\bigr)^2 & \text{if}\quad \tfrac{4}{d}<\alpha < \tfrac{d}{d+2}\cdot\tfrac{4}{d-2}, \\
a>-\bigl(\tfrac{d-2}{2}\bigr)^2 +\bigl(\tfrac{d-2}{2}-\tfrac{1}{\alpha}\tfrac{2d}{d+2}\bigr)^2 &\text{if}\quad \tfrac{d}{d+2}\cdot\tfrac{4}{d-2}\leq\alpha<\tfrac{4}{d-2}.
\end{cases}
\end{align}
The restrictions on $(d,a)$ in \eqref{important} stem from the fact that we work in $C_t H_a^1$. 
\end{remark}

\begin{remark}[Persistence of regularity]\label{remark1} Suppose $u:I\times\R^d\to\C$ is a solution to (NLSa) such that
\begin{equation}\label{suppose...}
\|u(t_0)\|_{H_x^1} \leq E\qtq{for some} t_0\in I\qtq{and}  \|u\|_{L_{t,x}^{q_0}(I\times\R^d)} \leq L.
\end{equation}
Then
\begin{equation}\label{in-fact}
\|u\|_{S_a^1(I)} \lesssim_{E,L} 1.
\end{equation}

\begin{proof}[Proof of \eqref{in-fact}]  In the following, we consider all norms over $I\times\R^d$.  We will consider separately the two cases in \eqref{important}.  We will show that in each case we may find exponents $q,q_1,r,r_1,r_2$ such that the following hold:
\begin{itemize}
\item[(i)] $(q,r)$ and $(q_1,r_1)$ are admissible pairs,
\item[(ii)] $\dot H_x^{s_c,r_1}\hookrightarrow L_x^{r_2},$
\item[(iii)] $\dot H_a^{s_c,r_1}$ and $\dot H_x^{s_c,r_1}$ are equivalent,
\item[(iv)] $\dot H_a^{1,r}$ and $\dot H_x^{1,r}$ are equivalent,
\item[(v)] the following nonlinear estimate holds by H\"older's inequality, the Sobolev embedding (ii), and the equivalence of Sobolev spaces in (iii) and (iv):
\begin{equation}\label{H1se}
\| |u|^\alpha u\|_{L_t^2 H_a^{1,\frac{2d}{d+2}}} \lesssim \|u\|_{L_t^{q_1}L_x^{r_2}}^\alpha \| u\|_{L_t^q H_x^{1,r}} \lesssim \|u\|_{L_t^{q_1}\dot H_a^{s_c,r_1}}^\alpha \|u\|_{L_t^q H_a^{1,r}}.
\end{equation}
\end{itemize}

If we can find such exponents, then using \eqref{suppose...} and a standard bootstrap argument (estimating the nonlinearity as in \eqref{critical-lwp}), one can first show that
\[
u\in L_t^\infty H_a^{s_c} \cap L_t^{q_0}H_a^{s_c,r_0}\cap L_t^{q_1}H_a^{s_c,r_1}. 
\]
To prove the $S_a^1$-estimates, one can use \eqref{H1se} to deduce that $\L^{\frac12}u$ belongs to every admissible Strichartz space other than the endpoint $L_t^2 L_x^{\frac{2d}{d-2}}$.  Perturbing the spaces slightly (without ruining equivalence of Sobolev spaces) then allows one to get the endpoint.

\textbf{1.}  First consider $\tfrac{4}{3}<\alpha\leq 2$ in $d=3$.  In this case, we choose
\[
q_1=2\alpha+,\quad r_2=3\alpha-,\quad r_1=\tfrac{6\alpha}{3\alpha-2}-,\quad (q,r)=(\infty-,2+).
\]
We do not simply choose $(q,r)=(\infty,2)$ in this case for the sake of an approximation argument later in the paper (cf. \eqref{embed-cc}). 
\textbf{2.}  We next consider $\tfrac{2}{d-2}\vee \tfrac{4}{d}<\alpha<\tfrac{4}{d-2}$ in dimensions $3\leq d\leq 6$.  We need to choose spaces a bit more delicately.  First define
\[
q_1=\tfrac{4\alpha^2}{2-\alpha(d-4)}, \quad r_2 = \tfrac{2d\alpha^2}{d\alpha-2}, \quad r_1 = \tfrac{2d\alpha^2}{d\alpha^2 - 2 +\alpha(d-4)}  
\]
Then $(q_1,r_2)$ is an admissible pair: As $\alpha<\frac{2}{d-4}$ (cf. the remarks surrounding \eqref{why}), we have $q_1<\infty$. The condition $q_1>2$ boils down to a quadratic equation for $\alpha$, resulting in the constraint
\[
\alpha>\alpha_0(d):=\tfrac{4-d+\sqrt{d^2-8d+32}}{4}.
\]
However, one can check that $\alpha_0(d)\leq \frac{4}{d}$ for $3\leq d\leq 6$.  Note also that by Sobolev embedding, we also have $\dot H_x^{s_c,r_1}\hookrightarrow L_x^{r_2}$.  Furthermore, under the contraints \eqref{important}, we have that $\dot H_x^{s_c,r_1}$ and $\dot H_a^{s_c,r_1}$ are equivalent. 

Finally, we let
\[
(q,r) = (\tfrac{4\alpha}{\alpha(d-2)-2}, \tfrac{d\alpha}{\alpha+1}). 
\]
This is an admissible pair; indeed $\tfrac{2}{d-2}<\alpha<\tfrac{2}{d-4}$ guarantees $2<q<\infty$.  Furthermore, the restrictions in \eqref{important} guarantee that $\dot H_a^{1,r}$ and $\dot H_x^{1,r}$ are equivalent. 

This completes the proof of \eqref{in-fact}. \end{proof}

As a consequence of \eqref{in-fact}, we have the following:
\begin{itemize}
\item[(i)] If the $L_{t,x}^{q_0}$-norm of a solution remains bounded throughout its lifespan, then the solution may be extended globally in time.
\item[(ii)] If the solution belongs to $L_{t,x}^{q_0}(\R\times\R^d)$, then the solution scatters in $H_a^1$.
\end{itemize}

Indeed, for (i) we need only note that in this case, $u$ remains uniformly bounded in $H_a^1$. For (ii), we can use Strichartz and estimate as in \eqref{H1se} to show that $\{e^{it\L}u(t)\}_t$ is Cauchy in $H_a^1$ as $t\to\pm\infty$.
\end{remark}

We next record a stability result for \eqref{nls}, which will play an important role in the proofs of Theorems~\ref{T:embedding} and~\ref{T:exist}.  The proof is standard and relies on the estimates used above; thus, we omit the proof.

\begin{theorem}[Stability]\label{T:stab}  Suppose $(d,a,\alpha)$ satisfy \eqref{important}. Let $I$ be a compact time interval and let $\tilde{v}$ be an approximate solution to \eqref{nls} on $I\times\R^d$ in the sense that
\[
(i\partial_t-\L)\tilde{v} =\pm|\tilde{v}|^{\alpha}\tilde{v} + e, \quad v(t_0)= \tilde v_0\in H_x^1(\R^d)
\]
for some suitable small function $e:I\times\R^d\to\C$. Fix $v_0\in H^1_x$ and assume that for some constants $E,L>0$ we have
\[
\|v_0\|_{H^1_x} + \|\tilde{v}_0\|_{H_x^1} \leq E \qtq{and} \| \tilde{v}\|_{L_{t,x}^{q_0}} \leq L
\]
There exists $\eps_0=\eps_0(E,L)>0$ such that if $0<\eps<\eps_0$ and
\begin{equation}\label{stab:small}
\|\tilde v_0 - v_0 \|_{\dot H_x^{s_c}} + \||\nabla|^{s_c} e\|_{N(I)} < \eps,
\end{equation}
where
\[
N(I) := L_t^1 L_x^2+ L_t^\rho L_x^\gamma+L_t^{\rho+} L_x^{\gamma-} + L_t^{2-}L_x^{\frac{2d}{d+2}+},
\]
then there exists a solution $v:I\times\R^d\to\C$ to \eqref{nls} with $v(t_0)=v_0$ satisfying
\begin{align}
\label{E:stab}
&\|v-\tilde v\|_{\dot S_a^{s_c}(I)} \lesssim_{E,L} \eps, \\
\label{E:stab-bound}
&\| v\|_{S_a^1(I)}\lesssim_{E,L} 1.
\end{align}

Additionally, if
\[
\|\tilde v_0 - v_0 \|_{\dot H_x^{s_c+}} + \||\nabla|^{s_c+} e\|_{N(I)} < \eps,
\]
then we  have
\begin{align}\label{E:stab2}
\| v-\tilde v\|_{\dot S_a^{s_c+}(I\times\R^d)}\lesssim_{E,L}\eps.
\end{align}

\end{theorem}

\subsection{Virial identities}\label{S:virial}
In this section, we recall some standard virial identities.  Given a weight $w:\R^d\to\R$ and a solution $u$ to \eqref{nls}, we define
\[
V(t;w) := \int |u(t,x)|^2 w(x)\,dx.
\]
A direct computation yields
\begin{equation}\label{virial}
\begin{aligned}
& \partial_t V(t;w) = \int 2\Im \bar u \nabla u \cdot \nabla w \,dx, \\
& \partial_{tt}V(t;w) = \int (-\Delta\Delta w)|u|^2 + 4\Re \bar u_j u_k w_{jk} + 4|u|^2 \tfrac{ax}{|x|^4}\cdot \nabla w + \tfrac{2\mu\alpha}{\alpha+2}|u|^{\alpha+2} \Delta w\,dx,
\end{aligned}
\end{equation}
where subscripts denote partial derivatives and repeated indices are summed. 

Choosing $w(x)=|x|^2$ or a truncated version thereof, one arrives at the following.
\begin{lemma}[Virial identities]\label{L:virial0} Let $u$ solve \eqref{nls}. The following hold:
\begin{itemize}
\item Choosing $w(x)=|x|^2$,  
\[
\partial_{tt} V(t;|x|^2) = 8\|u(t)\|_{\dot H_a^1}^2 + \tfrac{4\mu\alpha d}{\alpha+2}\|u(t)\|_{L_x^{\alpha+2}}^{\alpha+2}.
\]
\item Let $w_R(x) = R^2\phi(\frac{x}{R})$, where $R>1$ and $\phi$ is a smooth, non-negative radial function satisfying
\begin{equation}\label{phi-virial}
\phi(x)=\begin{cases} |x|^2 & |x|\leq 1 \\ 9 & |x|>3,\end{cases} \qtq{with}  |\nabla \phi|\leq 2|x|,\quad |\partial_{jk}\phi|\leq 2.
\end{equation}
Then we have
\begin{align}\nonumber \partial_{tt}& V(t;w_R)\\
\nonumber
  &= 8\Bigl[\|u(t)\|_{\dot H_a^1}^2 + \tfrac{\mu\alpha d}{2(\alpha+2)} \|u(t)\|_{L_x^{\alpha+2}}^{\alpha+2}\Bigr]  \\
\label{virial-sign}
& \quad  + 4\int_{|x|>R} \Re \bar u_j u_k\partial_{jk}[w_R]  + |u|^2 \tfrac{ax}{|x|^4}\cdot \nabla w_R \,dx - 8\int_{|x|>R} |(\L)^{\frac12} u|^2 \,dx \\
\nonumber &\quad  +O\biggl( \int_{|x|\geq R} R^{-2}|u|^2 + |u|^{\alpha+2}
\,dx\biggr).
\end{align}
\end{itemize}
\end{lemma}

\subsection{Variational analysis}\label{S:var}
In this section, we discuss the variational analysis related to the sharp Gagliardo--Nirenberg inequality:
\begin{equation}\tag{\ref{E:GN}}
\|f\|_{L_x^{\alpha+2}}^{\alpha+2} \leq C_a \|f\|^{\frac{4-(d-2)\alpha}{2}}_{L_x^2}\|f\|_{\dot H_a^1}^{\frac{d\alpha}{2}}.
\end{equation}

\begin{theorem}[Sharp Gagliardo--Nirenberg inequality]\label{T:GN} Fix $\mu=-1$, $a>-(\tfrac{d-2}{2})^2$ and define
\[
C_a := \sup\bigl\{ \|f\|_{L_x^{\alpha+2}}^{\alpha+2} \div \bigl[\|f\|^{\frac{4-(d-2)\alpha}{2}}_{L_x^2} \|f\|_{\dot H_a^1}^{\frac{d\alpha}{2}}\bigr]:f\in H_a^1\backslash\{0\}\ \bigl\}.
\]
Then $C_a\in(0, \infty)$ and the following hold:
\begin{itemize}
\item[(i)] If $a\leq 0$, then equality in  inequality \eqref{E:GN} is attained by a function $Q_a\in H_a^1$, which is a non-zero, non-negative, radial solution to the elliptic problem
\begin{equation}
\label{elliptic}
-\L Q_a - Q_a + Q_a^{\alpha+1} = 0.
\end{equation}
\item[(ii)] If $a>0$, then $C_a = C_0$, but equality in \eqref{E:GN} is never attained.
\end{itemize}
\end{theorem}
Mutatis mutandis, the proof of Theorem~\ref{T:GN} is the same as the proof appearing in \cite[Section~3]{KVMZ}, and thus we omit it.

Now fix $a\leq 0$ and let $Q_a$ be as in Theorem~\ref{T:GN}.  Multiplying \eqref{elliptic} by $Q_a$ and $x\cdot\nabla Q_a$ and integrating leads to the Pohozaev identities
\begin{align*}
\| Q_a\|_{\dot H_a^1}^2+  \|Q_a\|_{L_x^2}^2 - \|Q_a\|_{L_x^{\alpha+2}}^{\alpha+2} = \tfrac{d-2}{2} \| Q_a\|_{\dot H_a^1}^2 + \tfrac{d}{2}\|Q_a\|_{L_x^2}^2 - \tfrac{d}{\alpha+2} \|Q_a\|_{L_x^{\alpha+2}}^{\alpha+2} = 0.
\end{align*}
In particular,
\begin{equation}\label{poho}
\|Q_a\|_{L_x^2}^2 = \tfrac{4-\alpha(d-2)}{\alpha d}\|Q_a\|_{\dot H_a^1}^2 = \tfrac{4-\alpha(d-2)}{2(\alpha+2)} \|Q_a\|_{L_x^{\alpha+2}}^{\alpha+2}
\end{equation}
and
\begin{equation}\label{var-ca}
C_a = \tfrac{2(\alpha+2)[4-\alpha(d-2)]^{\frac{d\alpha}{4}-1}}{(\alpha d)^{\frac{d\alpha}{4}}} \|Q_a\|_{L_x^2}^{-\alpha}.
\end{equation}

We define
 \begin{align}\label{eaka}
  \E_a := \tfrac{(\alpha d-4)(\alpha d)^{\frac{d\alpha}{4-d\alpha}}}{2(2(\alpha+2))^{\frac{4}{4-d\alpha}}} C_a^{\frac{4}{4-d\alpha}},\quad\quad
  \K_a := \left(\tfrac{2(\alpha+2)}{\alpha d}\right)^{\frac{2}{d\alpha-4}} C_a^{\frac{2}{4-d\alpha}}.
 \end{align}
One can check that 
\begin{equation}\label{thresholds-Q}
\E_a := M(Q_{a\wedge 0})^{\sigma}E_{a\wedge 0}(Q_{a\wedge 0}) \qtq{and}  \K_a := \|Q_{a\wedge 0}\|^{\sigma}_{L_x^2}\|Q_{a\wedge 0}\|_{\dot H_{a\wedge 0}^1},
\end{equation}
where $\sigma=\tfrac{1}{s_c}-1$.
\begin{corollary}[Comparison of thresholds] \label{C:thresholds} Assume $\mu=-1$, then for any $a>-(\tfrac{d-2}{2})^2$, we have
\[
\E_a \leq \E_0\qtq{and} \K_a\leq \K_0.
\]
\end{corollary}
\begin{proof} When $a\geq 0$, we have $\E_a=\E_0$ and $\K_a=\K_0$ by definition.

For $a<0$, we note
\[
\|Q_0\|_{L_x^{\alpha+2}}^{\alpha+2} = C_0\|Q_0\|^{\frac{4-(d-2)\alpha}{2}}_{L_x^2} \|Q_0\|_{\dot H^1}^{\frac{d\alpha}{2}} > C_0\|Q_0\|^{\frac{4-(d-2)\alpha}{2}}_{L_x^2} \|Q_0\|_{\dot H_a^1}^{\frac{d\alpha}{2}}.
\]
This implies $C_0 < C_a$.  The result follows.
\end{proof}

The following proposition connects the sharp Gagliardo--Nirenberg inequality with the quantities appearing in the virial identities.

\begin{proposition}[Coercivity]\label{P:coercive}Let $\mu=-1$ and $a>-(\tfrac{d-2}{2})^2$. Let $u:I\times\R^d\to\C$ be the maximal-lifespan solution to \eqref{nls} with $u(t_0)=u_0\in H_a^1\backslash\{0\}$ for some $t_0\in I$. Assume that
\begin{equation}\label{quant-below}
M(u_0)^{\sigma}E_a(u_0) \leq (1-\delta)\E_a\qtq{for some}\delta>0.
\end{equation}
Then there exist $\delta'=\delta'(\delta)>0$, $c=c(\delta,a,\|u_0\|_{L_x^2})>0$, and $\eps=\eps(\delta)>0$ such that:
\begin{itemize}
\item[a.]If $\|u_0\|^{\sigma}_{L_x^2} \| u_0\|_{\dot H_a^1} \leq \K_a$, then for all $t\in I$,
\begin{itemize}
\item[(i)] $\|u(t)\|^{\sigma}_{L_x^2} \|u(t)\|_{\dot H_a^1} \leq (1-\delta')\K_a$,
\item[(ii)] $\|u(t)\|_{\dot H_a^1}^2 - \tfrac{\alpha d}{2(\alpha+2)} \|u(t)\|_{L_x^{\alpha+2}}^{\alpha+2} \geq c\|u(t)\|_{\dot H_a^1}^2$
\item[(iii)] $\tfrac12\|u(t)\|_{\dot H_a^1}^2 [ 1 - \tfrac{2}{\alpha+2} C_a\K_a(1-\delta')] \leq E_a(u) \leq \tfrac12 \| u(t)\|_{\dot H_a^1}^2,$
\end{itemize}
\item[b.] If $\|u_0\|^{\sigma}_{L_x^2} \|u_0\|_{\dot H_a^1} \geq \K_a$, then for all $t\in I$,
\begin{itemize}
\item[(i)] $\|u(t)\|^{\sigma}_{L_x^2} \|u(t)\|_{\dot H_a^1} \geq (1+\delta')\K_a$,
\item[(ii)] $(1+\eps)\|u(t)\|_{\dot H_a^1}^2 - \tfrac{\alpha d}{2(\alpha+2)}\|u(t)\|_{L_x^{\alpha+2}}^{\alpha+2} \leq -c < 0.$
\end{itemize}
\end{itemize}
\end{proposition}

\begin{proof}
 By the sharp Gagliardo--Nirenberg inequality, conservation of mass and energy, and \eqref{quant-below}, a.(i) and b.(i) follow from a continuity argument.
For claim a.(iii), the upper bound is trivial, since the nonlinearity is focusing.

For the lower bound, by the sharp Gagliardo--Nirenberg inequality, a.(i)  and \eqref{eaka}, we have
\begin{align*}
E_a(u) &  \geq \tfrac12\|u(t)\|_{\dot H_a^1}^2 [ 1 - \tfrac{2}{\alpha+2} C_a \|u(t)\|^{\frac{4-\alpha(d-2)}{2}}_{L_x^2} \|u(t) \|^{\frac{d\alpha}{2}}_{\dot H_a^1} ] \\
& \geq  \tfrac12\|u(t)\|_{\dot H_a^1}^2 [ 1 - \tfrac{2}{\alpha+2} C_a\K_a(1-\delta')]
\end{align*}
for all $t\in I$. Thus a.(iii) holds.

For a.(ii) and b.(ii), note that
\begin{align*}
\|u(t)\|_{\dot H_a^1}^2 - \tfrac{\alpha d}{2(\alpha+2)} \|u(t)\|_{L_x^{\alpha+2}}^{\alpha+2}& = \tfrac{\alpha d}{2} E_a(u) - \tfrac{\alpha d-4}{4} \|u(t)\|_{\dot H_a^1}^2, \\
(1+\eps)\|u(t)\|_{\dot H_a^1}^2 - \tfrac{\alpha d}{2(\alpha+2)} \|u(t)\|_{L_x^{\alpha+2}}^{\alpha+2}&= \tfrac{\alpha d}{2} E_a(u) - (\tfrac{\alpha d-4}{4}-\eps)\|u(t)\|_{\dot H_a^1}^2,
\end{align*}
for $t\in I$, where $\eps>0$ will be chosen below. Thus a.(ii) follows from a.(iii) by choosing any $0<c\leq \delta'$.

 For b.(ii), by the conservation of mass and energy, \eqref{quant-below}, \eqref{eaka}, and
b.(i), we have
\begin{align*}
\tfrac{\alpha d}{2}E_a(u) - (\tfrac12-\eps)\|u(t)\|_{\dot H_a^1}^2 &\leq \tfrac{1}{M(u)^{\sigma}}\bigl[\tfrac{\alpha d}{2}\E_a - (\tfrac12-\eps)(1+\delta')^2 \K_a^2\bigr]
\end{align*}
provided $\eps$ is sufficiently small depending on $\delta'$.  Thus b.(ii) follows.  \end{proof}

\begin{corollary}\label{R:coercive} Let $\mu=-1$ and suppose $u_0\in H^1$ satisfies $M(u_0)^{\sigma}E_a(u_0)<\E_a$ and $\|u_0\|^{\sigma}_{L_x^2} \|u_0\|_{\dot H_a^1} \leq \K_a.$ Then the corresponding solution to \eqref{nls} is global-in-time.  For $\mu=1$, all solutions with $H^1$ data are global.
\end{corollary}

\begin{proof} For the focusing case, the solution $u$ to \eqref{nls} with initial data $u_0$ obeys 
\[
\|u(t)\|^{\sigma}_{L_x^2} \|u(t)\|_{\dot H_a^1} < \K_a
\]
for all $t$ in the lifespan of $u$.  In particular, $u$ remains uniformly bounded in $H_x^1$, and hence by Theorem~\ref{T:LWP} may be extended globally in time.  For the defocusing case, we simply rely on conservation of mass and energy.
\end{proof}

\section{Concentration compactness}\label{S:cc}
In this section, we prove a linear profile decomposition adapted to the $H^1\to L_{t,x}^{q_0}$ Strichartz inequality. We further prove a result concerning the embedding of nonlinear profiles, which will be used in the proof of the existence of minimal blowup solutions (Theorem~\ref{T:exist}).  Recall from Section~\ref{S:FS} that $q_0:=\frac{\alpha(d+2)}{2}$. 

\subsection{Linear profile decomposition}
\begin{proposition}[Linear profile decomposition]\label{P:LPD} Fix $a>-(\tfrac{d-2}{2})^2$ and let $\{f_n\}$ be a bounded sequence in $H_a^1(\R^d)$. Passing to a subsequence, there exist $J^*\in \{0,1,2,\dots,\infty\}$, functions $\{\phi^j\}_{j=1}^{J^*}\subset H_x^1(\R^d)$, and $\{(t_n^j,x_n^j)\}_{j=1}^{J^*}\subset\R\times\R^d$ satisfying the following:  for each finite $0\leq J\leq J^*$,
\begin{equation}\label{E:LPD}
f_n=\sum_{j=1}^J \phi_n^j + r_n^J,
\end{equation}
where $\phi_n^j = [e^{it_n^j \L^{n_j}}\phi^j](x-x_n^j)$ and $\L^{n_j}$ is as in Definition~\ref{D:ops}, corresponding to the sequence $\{x_n^j\}_{n=1}^\infty$.

The remainder $r_n^J$ satisfies
\begin{equation}\label{rnJweak}
\bigl(e^{-it_n^J\L} r_n^J\bigr)(x+x_n^J) \rightharpoonup 0 \qtq{weakly in}  H_x^1
\end{equation}
and
\begin{equation}\label{rnJ}
\lim_{J\to J^*}\limsup_{n\to\infty}\|e^{-it\L}r_n^J\|_{L_{t,x}^{q_0}(\R\times\R^d)} = 0.
\end{equation}

The parameters $(t_n^j,x_n^j)$ are asymptotically orthogonal: for any $j\neq k$,
\begin{equation}\label{orthogonal}
\lim_{n\to\infty} \bigl(|t_n^j - t_n^k| + |x_n^j - x_n^k| \bigr)= \infty.
\end{equation}

Furthermore, for each $j$, we may assume that either $t_n^j\to\pm\infty$ or $t_n^j\equiv 0$, and either $|x_n^j|\to\infty$ or $x_n^j\equiv 0$.

Moreover, for each finite $0\leq J\leq J^*$ we have the following asymptotic orthogonality:
\begin{align}
\label{decouple1}
\lim_{n\to\infty} \bigl\{ \|(\L)^{\frac{s}{2}} f_n\|_{L_x^2}^2- \sum_{j=1}^J \|(\L)^{\frac{s}{2}}\phi_n^j\|_{L_x^2}^2 - \|(\L)^{\frac{s}{2}}r_n^J\|_{L_x^2}^2 \bigr\}&=0,\ s\in\{0,1\},\\
\label{decouple2}
\lim_{n\to\infty} \bigl\{\| f_n\|_{L_x^{\alpha+2}}^{\alpha+2} - \sum_{j=1}^J \|\phi_n^j\|_{L_x^{\alpha+2}}^{\alpha+2} - \| r_n^J\|_{L_x^{\alpha+2}}^{\alpha+2}\bigr\}&=0.
\end{align}
\end{proposition}

We start with a refined Strichartz estimate.

\begin{lemma}[Refined Strichartz]\label{L:RS}
Let $(d,a,\alpha)$ satisfy \eqref{important}.  There exists $\theta=\theta(d,a,\alpha)$ such that
\begin{equation}\label{refined}
\| e^{-it\L}f\|_{L_{t,x}^{q_0}(\R\times\R^d)} \lesssim \|f\|_{\dot H_a^{s_c}}^\theta \sup_{N\in 2^{\Z}}\|e^{-it\L}f_N\|_{L_{t,x}^{q_0}(\R\times\R^d)}^{1-\theta}.
\end{equation}

\end{lemma}
\begin{proof} To shorten formulas, we set $u(t) = e^{-it\L}f$ and denote frequency projections with subscripts.  All space-time norms are taken over $\R\times\R^d$.

We break into two cases. 

\textbf{1.}  First suppose $\alpha>\frac{8}{d+2}$, so that $q_0>4$. (In light of \eqref{intercritical}, this restricts to dimensions $3\leq d\leq 5$.)  We recall also the exponent $r_0$ defined in Section~\ref{S:FS}.  By the square function estimate (Lemma~\ref{T:sq}), Bernstein, Strichartz, and Cauchy--Schwarz, we may estimate
\begin{align*}
\|u\|_{L_{t,x}^{q_0}} &\lesssim \iint\biggl[\sum_N |u_N|^2\biggr]^{\frac{q_0}{2}}\,dx\,dt \\
& \lesssim \sum_{N_1\leq N_2} \iint\biggl[\sum_N |u_N|^2\biggr]^{\frac{q_0}{2}-2} |u_{N_1}|^2|u_{N_2}|^2\,dx\,dt \\
& \lesssim \|u\|_{L_{t,x}^{q_0}}^{q_0-4}\sum_{N_1\leq N_2} \|u_{N_1}\|_{L_t^{q_0}L_x^{q_0+}}\|u_{N_2}\|_{L_t^{q_0}L_x^{q_0-}}\prod_{j=1}^2 \|u_{N_j}\|_{L_{t,x}^{q_0}} \\
& \lesssim \|f\|_{\dot H_a^{s_c}}^{q_0-4} \sup_N \|u_N\|_{L_{t,x}^{q_0}}^2 \sum_{N_1\leq N_2} \bigl(\tfrac{N_1}{N_2}\bigr)^{0+}\prod_{j=1}^2 \|u_{N_j}\|_{L_t^{q_0}\dot H_a^{s_c,r_0}} \\
& \lesssim \|f\|_{\dot H_a^{s_c}}^{q_0-4}\sup_N\|u_N\|_{L_{t,x}^{q_0}}^2 \sum_{N_1\leq N_2}\bigl(\tfrac{N_1}{N_2}\bigr)^{0+}\|f_{N_1}\|_{\dot H_a^{s_c}}\|f_{N_2}\|_{\dot H_a^{s_c}} \\
& \lesssim \|f\|_{\dot H_a^{s_c}}^{q_0-2}\sup_N\|e^{-it\L}f_N\|_{L_{t,x}^{q_0}}^2.
\end{align*}
The result follows in this case. 

\textbf{2.}  Suppose $\alpha<\tfrac{8}{d+2}$, so that $2<q_0<4$ (this is always the case for $d\geq 6$).  Estimating in a similar fashion to case one, we find
\begin{align*}
\|u\|_{L_{t,x}^{q_0}}^{q_0} &\lesssim \iint \biggl[\sum_N |u_N|^2\biggr]^{\frac{q_0}{2}}\,dx\,dt \\
& \lesssim \iint \biggl[\sum_N |u_N|^{\frac{q_0}{2}}\biggr]^2\,dx\,dt \\
& \lesssim\sum_{N_1\leq N_2} \iint |u_{N_1}|^{\frac{q_0}{2}}|u_{N_2}|^{\frac{q_0}{2}} \,dx \,dt \\
& \lesssim\sum_{N_1\leq N_2} \|u_{N_1}\|_{L_t^{q_0}L_x^{q_0+}}\|u_{N_2}\|_{L_t^{q_0}L_x^{q_0-}} \prod_{j=1}^2 \|u_{N_j}\|_{L_{t,x}^{q_0}}^{\frac{q_0}{2}-1} \\
& \lesssim \sup_N \|e^{-it\L}f_N\|_{L_{t,x}^{q_0}}^{q_0-2}\|f\|_{\dot H_a^{s_c}}^2,
\end{align*}
giving the result in this case.\end{proof}

We next prove an inverse Strichartz inequality.

\begin{proposition}[Inverse Strichartz]\label{P:IS} Let $(d,a,\alpha)$ satisfy \eqref{important}. Suppose $\{f_n\}\subset H^1_a(\R^d)$ satisfy
\[
\lim_{n\to\infty} \|f_n\|_{H^1_a} = A<\infty \qtq{and} \lim_{n\to\infty} \| e^{-it\L} f_n\|_{L_{t,x}^{q_0}} = \eps > 0.
\]
Up to a subsequence, there exist $\phi\in H_x^1$ and $\{(t_n,x_n)\}\subset\R\times\R^d$ such that
\begin{align}
\label{weak}
& g_n(\cdot) = [e^{-it_n\L} f_n](\cdot+x_n) \rightharpoonup \phi(\cdot) \qtq{weakly in}  H_x^1,  \\
\label{LB} & \|\phi\|_{H_a^1} \gtrsim \eps(\tfrac{\eps}{A})^{c}.
\end{align}

Furthermore, defining
\[
\phi_n(x) = e^{it_n \L}[\phi(\cdot-x_n)](x)=[e^{it_n\L^n}\phi](x-x_n),
\]
with $\L^n$ as in Definition~\ref{D:ops}, we have
\begin{align}
& \lim_{n\to\infty} \bigl\{ \|(\L)^{\frac{s}{2}} f_n\|_{L_x^2}^2 - \|(\L)^{\frac{s}{2}} (f_n-\phi_n)\|_{L_x^2}^2 - \|(\L)^{\frac{s}{2}} \phi_n\|_{L_x^2}^2 \bigr\} =0,\ s\in\{0,1\}, \label{H1-decouple} \\
& \lim_{n\to\infty} \bigl\{ \|f_n\|_{L_x^{\alpha+2}}^{\alpha+2} - \|f_n-\phi_n\|_{L_x^{\alpha+2}}^{\alpha+2}  - \|\phi_n\|_{L_x^{\alpha+2}}^{\alpha+2}\bigr\} = 0.    \label{L4-decouple}
\end{align}

Finally, we may assume that either $t_n\to\pm\infty$ or $t_n\equiv 0$, and either $|x_n|\to\infty$ or $x_n\equiv 0$.
\end{proposition}

\begin{proof} Throughout the proof, we let $c$ denote small positive constants whose precise values do not play any important role; in particular, this constant may change from line to line.

By Lemma~\ref{L:RS}, for $n$ sufficiently large, there exists $N_n\in 2^{\mathbb{Z}}$ such that
\[
\| e^{-it\L}  P_{N_n}^a f_n \|_{L_{t,x}^{q_0}} \gtrsim \eps(\tfrac{\eps}{A})^{c}.
\]
By Bernstein and Strichartz estimates, we have
\[
\| e^{-it\L}  P_N^a f_n \|_{L_{t,x}^{q_0}} \lesssim (N^{s_{c}}\vee N^{1-s_{c}}) A\qtq{for any}  N\in 2^{\mathbb{Z}},
\]
so that we must have
\[
(\tfrac{\eps}{A})^{c'} \lesssim N_n \lesssim (\tfrac{A}{\eps})^{c''}.
\]
Passing to a subsequence, we may assume $N_n\equiv N_*$. Thus
\[
\| e^{-it\L} P_{N_*}^a f_n \|_{L_{t,x}^{q_0}} \gtrsim \eps (\tfrac{\eps}{A})^{c}
\]
for all $n$ sufficiently large.  In what follows, we use the shorthand:
\[
 P_* :=  P_{N_*}^a.
\]

Note that by H\"older and Bernstein inequalities, for any $N>0$ and $C>0$ we have
\[
\| P_N^a F\|_{L_x^{q_0}(\{|x|\leq C N^{-1}\})} \lesssim \| P_N^a F\|_{L_x^{q_0+}} \|1\|_{L_x^{\infty-}(\{|x|\leq C N^{-1}\})}
    \lesssim C^{0+} \|P_N^a F\|_{L_x^{q_0}}.
\]
Thus,
\[
\|e^{-it\L} P_* f_n \|_{L_{t,x}^{q_0}(\R\times\{|x|\geq \alpha \})}\gtrsim \eps(\tfrac{\eps}{A})^{c},\qtq{provided}\alpha = C N_*^{-1}
\]
for $C>0$ sufficiently small. Using this together with H\"older, Strichartz, and Bernstein inequalities, we find
\begin{align*}
\eps(\tfrac{\eps}{A})^{c} & \lesssim \| e^{-it\L} P_* f_n \|_{L_{t,x}^\infty(\R\times \{|x|\geq\alpha\})}^{\frac{\alpha d-4}{\alpha d}} \|e^{-it\L} P_* f_n \|_{L_{t,x}^{\frac{2(d+2)}{d}}}^{\frac{4}{\alpha d}}  \\
& \lesssim \| e^{-it\L} P_* f_n \|_{L_{t,x}^\infty(\R\times \{|x|\geq\alpha \})}^{\frac{\alpha d-4}{\alpha d}}
    \| f_n\|_{L_x^2}^{\frac{4}{\alpha d}},
\end{align*}
and hence there exist $(\tau_n,x_n)\in\R\times\R^d$ with $|x_n|\geq \alpha$ such that
\begin{equation}\label{bubble}
 \biggl| [e^{-i\tau_n\L}  P_* f_n](x_n)\biggr| \gtrsim \eps (\tfrac{\eps}{A})^{c}.
\end{equation}
Passing to a subsequence, we may assume $\tau_n\to\tau_\infty\in [-\infty,\infty]$. If $\tau_\infty\in\R$, we set $t_n\equiv 0$. If $\tau_\infty\in\{\pm\infty\}$, we set $t_n=\tau_n$. We may also assume that $x_n\to x_\infty\in\R^d\backslash\{0\}$ or $|x_n|\to\infty$.

We now let
\[
g_n(x) = [e^{-it_n\L} f_n](x+x_n),\qtq{i.e.}  f_n(x) = [e^{it_n\L^n}g_n](x-x_n),
\]
where $\L^n$ is as in Definition~\ref{D:ops}.  Note that
\[
\|g_n\|_{H_x^1} = \|e^{-it_n\L}f_n\|_{H_x^1}\sim \|e^{-it_n\L}f_n\|_{H_a^1} \lesssim A,
\]
so that $g_n$ converges weakly to some $\phi$ in $H_x^1$ (up to a subsequence). Define
\[
\phi_n(x) = e^{it_n\L}[\phi(\cdot-x_n)](x) = [e^{it_n\L^n}\phi](x-x_n).
\]
By a change of variables and weak convergence, we have
\begin{equation*}
\|f_n\|_{L_x^2}^2 - \|f_n - \phi_n \|_{L_x^2}^2 - \|\phi_n\|_{L_x^2}^2 = 2\Re\langle g_n, \phi\rangle - 2\langle\phi,\phi\rangle \to 0
\end{equation*}
as $n\to\infty$; using \eqref{coo1} as well, we get
\begin{align*}
\|f_n\|_{\dot H^1_a}^2 - \|f_n-\phi_n\|_{\dot H^1_a}^2 - \|\phi_n\|_{\dot H^1_a}^2 & = 2\Re\langle g_n,\L^n \phi\rangle - 2\langle\phi, \L^n \phi\rangle \\
& \to 2[\langle\phi,\L^\infty\phi\rangle - \langle\phi,\L^\infty\phi\rangle] = 0.
\end{align*}
This proves \eqref{H1-decouple}.

We next turn to \eqref{LB}. We define
\[
h_n = \begin{cases}
     P_*^n \delta_0 & \text{if }\tau_\infty \in \{\pm \infty\}, \\
    e^{-i\tau_n \L^n} P_*^n\delta_0 & \text{if }\tau_\infty\in\R,
    \end{cases}
\qtq{where}   P_*^n = e^{-\L^n/N_*^2} - e^{-4\L^n/N_*^2}.
\]
Note that after a change of variables, \eqref{bubble} reads
\[
\bigl|\langle h_n,g_n\rangle\bigr| \gtrsim \eps (\tfrac{\eps}{A})^{c}.
\]
As $|x_n|\geq \alpha>0$, we have by \eqref{coo2} and \eqref{coo3} that
\[
h_n \to \begin{cases}  P_*^\infty\delta_0 &\text{if } \tau_\infty\in\{\pm\infty\}, \\
    e^{-i\tau_\infty \L^\infty} P_*^\infty \delta_0 &\text{if } \tau_\infty\in\R,  \end{cases}
\qtq{where}   P_*^\infty = e^{-\L^\infty/N_*^2} - e^{-4\L^\infty/N_*^2}.
\]
Here the convergence holds strongly in $\dot H_x^{-1}$.  Thus, if $\tau_\infty\in\{\pm\infty\}$, we have
\[
\eps (\tfrac{\eps}{A})^{c} \lesssim |\langle  P_*^\infty \delta_0, \phi\rangle|\lesssim \|\phi\|_{L_x^2} \|  P_*^\infty \delta_0\|_{L_x^2}.
\]

By the heat kernel bounds of Lemma~\ref{L:kernel},  we can bound
\[
\|  P_*^\infty \delta_0\|_{L_x^2}  \lesssim N_*^{\frac{d}{2}}\lesssim  (\tfrac{A}{\eps})^{c},
\]
which implies
\[
\| \phi \|_{L_x^2} \gtrsim \eps(\tfrac{\eps}{A})^{c}.
\]
The case of $\tau_\infty\in\R$ is similar.  This proves \eqref{LB}.

We now turn to \eqref{L4-decouple}. Using Rellich--Kondrashov and passing to a subsequence, we may assume $g_n\to\phi$ almost everywhere.  Thus, Lemma~\ref{L:RF} implies
\[
\|g_n\|_{L_x^{\alpha+2}}^{\alpha+2} - \|g_n-\phi\|_{L_x^{\alpha+2}}^{\alpha+2} - \|\phi\|_{L_x^{\alpha+2}}^{\alpha+2} \to 0.
\]
This, together with a change of variables, gives \eqref{L4-decouple} in the case $t_n\equiv 0$.  If instead $t_n\to\pm\infty$, then \eqref{L4-decouple} follows from Corollary~\ref{L4-to-zero}.

Finally, if $x_n\to x_\infty\in\R^d$, then we may take $x_n\equiv 0$ by replacing $\phi(\cdot)$ with $\phi(\cdot-x_\infty)$.  By the continuity of translation in the strong $H_x^1$-topology,  all the conclusions still hold.
\end{proof}

With Proposition~\ref{P:IS} in place, the proof of Proposition~\ref{P:LPD} follows from a fairly standard inductive argument.  We omit the proof, referring the reader to similar proofs appearing in \cite{KVMZ, Visan}.  To prove orthogonality of the parameters requires two additional ingredients, which we state here without proof (see \cite{KMVZZ2}):

\begin{lemma}\label{con1} Let $f_n\in\dot H_x^1$ satisfy $f_n\rightharpoonup 0$ weakly in $\dot H_x^1$, and suppose $\tau_n\to\tau_0\in\R$. Then for any sequence $\{y_n\}\subset\R^d$, we have
\[
e^{-i\tau_n\L^n} f_n\rightharpoonup 0 \qtq{weakly in}  \dot H_x^1.
\]
Here $\L^n$ is as in Definition~\ref{D:ops}, corresponding to the sequence $\{y_n\}$.
\end{lemma}
\begin{lemma}\label{con2} Let $f\in \dot H_x^1$. Let $\{(\tau_n,x_n)\}\subset\R\times\R^d$ and suppose that either $|\tau_n|\to\infty$ or $|x_n|\to\infty$. Then for any sequence $\{y_n\}\subset\R^d$, we have
\[
[e^{-i\tau_n\L^n}f](\cdot+x_n)\rightharpoonup 0 \qtq{weakly in} \dot H_x^1,
\]
where $\L^n$ is as in Definition~\ref{D:ops}, corresponding to the sequence $\{y_n\}$.
\end{lemma}

\subsection{Embedding nonlinear profiles}
Central to the proof of existence of minimal blowup solutions will be a `nonlinear profile decomposition'.  The following result is essential to the construction of nonlinear solutions associated to profiles whose spatial translation parameter tends to infinity.  The idea is that such solutions should not be strongly affected by the potential, and hence may be approximated by solutions to \eqref{nls0}.  In particular, we rely on the results of \cite{cazenave, DHR, GV} for the free NLS (cf. Theorem~\ref{T:DHR1}). 

We will also need to approximate the nonlinear solutions by functions that are compactly supported in space-time.  We need to do this in several topologies, which we introduce here.  We recall that $q_0=\tfrac{\alpha(d+2)}{2}$.  We also recall the exponents $(q,r)$ and $(q_1,r_1)$ introduced in Remark~\ref{remark1}. We consider a parameter $s_c<s<1$, where take $s$ extremely close to $s_c$ (cf. the last statement in Theorem~\ref{T:stab}). We will prove approximation in the spaces
\begin{equation}\label{Xin}
L_{t,x}^{q_0},\quad L_{t,x}^{\frac{2(d+2)}{d}},\quad L_t^{q_1} \dot H_a^{s_c,r_1},\qtq{and} L_t^q \dot H_a^{s,r}. 
\end{equation}

\begin{proposition}[Embedding of nonlinear profiles]\label{T:embedding} Let $(d,a,\alpha)$ satisfy \eqref{important}. Let $\{t_n\}\subset\R$ satisfy $t_n\equiv 0$ or $t_n\to\pm\infty$, and let $\{x_n\}\subset\R^d$ satisfy $|x_n|\to\infty$. Let $\phi\in H_x^1(\R^d)$ and define
\[
\phi_n(x) = [e^{-it_n\L^n}\phi](x-x_n),
\]
where $\L^n$ is as in Definition~\ref{D:ops}. 

If $\mu=1$ (defocusing case), then for all $n$ sufficiently large, there exists a global solution $v_n$ to \eqref{nls} with $v_n(0)=\phi_n$ satisfying
\[
\| v_n\|_{S_a^1(\R)} \lesssim 1,
\]
with the implicit constant depending on $\|\phi\|_{H_x^1}$.

If $\mu=-1$ (focusing case), the same results hold provided
\begin{equation} \label{embedding-threshold}
\begin{aligned}
M(\phi)^{\sigma}E_{0}(\phi) < \E_0 \qtq{and} \|\phi\|^{\sigma}_{L_x^2}\|\phi\|_{\dot H_x^1} < \K_0
&\qtq{if}  t_n\equiv 0, \\
\tfrac12\|\phi\|_{L_x^2}^{2\sigma} \|\phi\|_{\dot H_x^1}^2 < \E_0 &\qtq{if}  t_n\to\pm\infty.
\end{aligned}
\end{equation}

In both scenarios, we have the following: for any $\eps>0$, there exist $N_\eps\in\mathbb{N}$ and $\psi_\eps\in C_c^\infty(\R\times\R^d)$ such that for $n\geq N_\eps$,
\begin{equation}\label{embed-cc}
\|v_n-\psi_\eps(\cdot+t_n,\cdot - x_n)\|_{X(\R\times\R^d)}  < \eta.
\end{equation}
for any $X$ appearing in \eqref{Xin}. 
\end{proposition}

\begin{proof}Note that
\begin{equation}\label{phin-bd}
\|\phi_n \|_{\dot H_x^1} \lesssim \|\phi\|_{\dot H_x^1}\qtq{uniformly in } n.
\end{equation}

\textbf{1.} We first find solutions to \eqref{nls0} related to $\phi$. Define $P_n = P_{\leq |x_n|^{\theta}}$ for some $0<\theta<1$. If $\mu=-1$, since $|x_n|\to\infty$, $P_n\phi$ satisfies \eqref{embedding-threshold} for all $n$ sufficiently large, so that we may use Theorem~\ref{T:DHR1}: If $t_n\equiv 0$, then we let $w_n$ and $w_\infty$ be the solutions to \eqref{nls0} with $w_n(0)=P_n\phi$ and $w_\infty(0) = \phi$; if $t_n\to\pm\infty$, we instead let $w_n$ and $w_\infty$ be the solutions to \eqref{nls0} satisfying
\[
\|w_n(t) - e^{it\Delta}P_n\phi\|_{H_x^1} \to 0 \qtq{and} \|w_\infty(t)-e^{it\Delta}\phi\|_{H_x^1}\to 0
\]
as $t\to\pm\infty$.  Note that in both cases, we have
\begin{equation}\label{w-good}
\|w_n\|_{S_0^1(\R)} + \|w_\infty\|_{S_0^1(\R)} \lesssim 1
\end{equation}
for $n$ sufficiently large, with the implicit constant depending on $\|\phi\|_{H_x^1}$.  For the defocusing case, we instead rely on the results of \cite{GV} (say) to construct $w_n,w_\infty$. Note that as $\|P_n\phi - \phi\|_{H_x^1}\to 0$ as $n\to\infty$, the stability theory for \eqref{nls0} implies that
\begin{equation}\label{embed-conv}
\lim_{n\to\infty} \| w_n - w_\infty \|_{S_0^1(\R)} = 0.
\end{equation}
By persistence of regularity for \eqref{nls0} and the fact that $\| |\nabla|^\lambda P_n \phi \|_{H_x^1} \lesssim |x_n|^{\lambda\theta}$ for any $\lambda\geq 0$, we have
\begin{equation}\label{embed-persist}
\| |\nabla|^{\lambda} w_n \|_{S_0^1(\R)} \lesssim  |x_n|^{\theta \lambda}\quad\text{for all }\lambda\geq 0\text{ and }n\text{ large}.
\end{equation}
Finally, note that in either case, $w_\infty$ scatters to some asymptotic states $w_\pm$ in $H_x^1$.

\textbf{2.} We next construct approximate solutions to \eqref{nls}. For each $n$, define $\chi_n$ to be a smooth function satisfying
\begin{equation}\label{chin}
\chi_n(x) = \begin{cases} 0 & |x+x_n| \leq \tfrac14 |x_n|, \\ 1 & |x+x_n| > \tfrac12 |x_n|, \end{cases}
\qtq{with} \sup_x |\partial^\alpha \chi_n (x) | \lesssim |x_n|^{-|\alpha|}
\end{equation}
for all multi-indices $\alpha$. Note that $\chi_n(x)\to 1$ as $n\to\infty$ for each $x\in\R^d$.  For $T>0$, we now define
\[
\tilde v_{n,T}(t,x) := \begin{cases} [\chi_n w_n](t,x-x_n) & |t| \leq T, \\
[e^{-i(t-T)\L}\tilde v_{n,T}(T)](x) & t > T, \\
[e^{-i(t+T)\L}\tilde v_{n,T}(-T)](x) & t<-T.
\end{cases}
\]

\textbf{3.} We are now in the position to construct $v_n$ by applying Theorem~\ref{T:stab}.  To do so, we must verify the following: For $\theta\in\{s_c,s\}$,
\begin{align}
\label{embed-bounds}
&\limsup_{T\to\infty}\limsup_{n\to\infty}\bigl\{\|\tilde v_{n,T}\|_{L_t^\infty H_x^1}+ \|\tilde v_{n,T} \|_{L_{t,x}^{q_0}}\bigr\} \lesssim 1, \\
\label{embed-data}
&\lim_{T\to\infty} \limsup_{n\to\infty} \| \tilde v_{n,T}(t_n) - \phi_n \|_{\dot H_x^\theta} = 0, \\
\label{embed-approx}
&\lim_{T\to\infty} \limsup_{n\to\infty} \||\nabla|^{\theta}[(i\partial_t - \L)\tilde v_{n,T} -\mu |\tilde v_{n,T}|^{\alpha} \tilde v_{n,T}] \|_{N(\R)} = 0,
\end{align}
where space-time norms are over $\R\times\R^d$.  For the definition of $N(\R)$, see Theorem~\ref{T:stab}.

Firstly, by Strichartz estimate, equivalence of Sobolev spaces, and \eqref{w-good}, we have
\begin{align}
\|\tilde v_{n,T} \|_{L_t^\infty H_x^1} & \lesssim \| \langle\nabla\rangle (\chi_n w_n) \|_{L_t^\infty L_x^2}\nonumber \\
&\lesssim  \|\nabla\chi_n\|_{L_x^d
}\|w_n\|_{L_t^\infty L_x^{\frac{2d}{d-2}}} + \|\chi_n\|_{L_x^\infty} \|\langle\nabla\rangle w_n\|_{L_t^\infty L_x^2}\nonumber  \\
&\lesssim \|w_n\|_{L_t^\infty H_x^1}\lesssim 1  \quad\text{uniformly in } n,T.\label{embed-bound0}
\end{align}
On the other hand, by Strichartz and \eqref{w-good}, we have
\[
\|\tilde v_{n,T} \|_{L_{t,x}^{q_0}} \lesssim 1\qtq{uniformly in} n,T.
\]
Hence \eqref{embed-bounds} holds.

Secondly, we prove \eqref{embed-data}. By \eqref{phin-bd} and \eqref{embed-bound0}, we first note that
\begin{equation}\label{embed-bound1}
\| \tilde v_{n,T}(t_n) - \phi_n \|_{\dot H_x^1} \lesssim 1\qtq{uniformly in } n,T.
\end{equation}

Consider the case $t_n\equiv 0$. Then
\[
\|\tilde v_{n,T}(0)-\phi_n\|_{L_x^2} = \|\chi_n P_n\phi - \phi\|_{L_x^2},
\]
which converges to zero as $n\to\infty$ by the dominated convergence theorem and Bernstein. Thus, by interpolation with \eqref{embed-bound1}, we see that \eqref{embed-data} holds when $t_n\equiv 0$.

Now consider the case $t_n\to\infty$; the case $t_n\to-\infty$ is similar.  For sufficiently large $n$, we have $t_n>T$, and hence (since $\L^\infty = -\Delta$)
\begin{align}
\nonumber
\|\tilde v_{n,T}(t_n) - \phi_n \|_{L_x^2} & = \|e^{iT\L^n}\chi_nw_n(T) - \phi\|_{L_x^2} \\
& \lesssim \|\chi_nw_n(T) - w_\infty(T) \|_{L_x^2} \label{e-data1} \\
& \quad + \|[e^{iT\L^n}-e^{iT\L^\infty}]w_\infty(T) \|_{L_x^2} \label{e-data2} \\
&\quad + \|e^{-iT\Delta}w_\infty(T) - \phi\|_{L_x^2}. \label{e-data3}
\end{align}
Using dominated convergence and \eqref{embed-conv}, we deduce that $\eqref{e-data1}\to 0$ as $n\to\infty$.  By \eqref{coo5}, we also find that $\eqref{e-data2}\to 0$ as $n\to\infty$.  Finally, by construction, we have that $\eqref{e-data3}\to 0$ as $T\to\infty$. Interpolating with \eqref{embed-bound1}, we see that \eqref{embed-data} holds in the case $t_n\to\pm\infty$, as well. This completes the proof of \eqref{embed-data}.

We now turn to \eqref{embed-approx}. First note that for $|t|>T$, we have that
\[
e_{n,T} := (i\partial_t - \L)\tilde v_{n,T} -\mu |\tilde v_{n,T}|^{\alpha} \tilde v_{n,T} = -\mu |\tilde v_{n,T}|^{\alpha} \tilde v_{n,T}.
\]
For $\theta\in\{s_c,s\}$, we estimate
\begin{align*}
\||\nabla|^\theta\bigl(|\tilde v_{n,T}|^{\alpha} \tilde v_{n,T}\bigr)\|_{L_t^\rho L_x^\gamma(\{t>T\}\times\R^d)} & \lesssim\|\tilde v_{n,T}\|_{L_{t,x}^{q_0}(\{t>T\}\times\R^d)}^{\alpha}
    \|\tilde v_{n,T}\|_{L_t^{q_0} \dot H_a^{{\theta},r_0}}  \\
&\lesssim \| e^{-it\L^n}[\chi_n w_n(T)]\|_{L_{t,x}^{q_0}((0,\infty)\times\R^d)}^{\alpha}.
\end{align*}
We now claim
\begin{equation}\label{embed-error1}
\lim_{T\to\infty}\limsup_{n\to\infty} \| e^{-it\L^n}[\chi_n w_n(T)]\|_{L_{t,x}^{q_0}((0,\infty)\times\R^d)} = 0,
\end{equation}
which implies \eqref{embed-approx} for times $t>T$.  (The case $t<-T$ is similar.) We use Sobolev embedding and Strichartz to estimate
\begin{align}
 \| e^{-it\L^n}[\chi_n w_n(T)]\|_{L_{t,x}^{q_0}((0,\infty)\times\R^d)} & \lesssim \|\chi_n w_n(T) - w_\infty(T)\|_{\dot H_x^{s_c}} \label{embed-error2} \\
&\quad + \|[e^{-it\L^n}-e^{-it\L^\infty}]w_\infty(T)\|_{L_{t,x}^{q_0}((0,\infty)\times\R^d)} \label{embed-error3} \\
&\quad + \|e^{it\Delta}[w_\infty(T) - e^{iT\Delta}w_+]\|_{L_{t,x}^{q_0}((0,\infty)\times\R^d)} \label{embed-error4} \\
&\quad + \|e^{it\Delta}w_+ \|_{L_{t,x}^{q_0}((T,\infty)\times\R^d)}. \label{embed-error5}
\end{align}
Note that it follows from $\dot H_x^1$-boundedness and our analysis of \eqref{e-data1} that $\eqref{embed-error2}\to 0$ as $n\to\infty$. Next, $\eqref{embed-error3}\to 0$ as $n\to\infty$ by Corollary~\ref{S-to-zero}. We have that $\eqref{embed-error4}\to 0$ as $T\to\infty$ by Strichartz and the definition of $w_+$. Finally, $\eqref{embed-error5}\to 0$ as $T\to\infty$ by monotone convergence theorem. This completes the proof of \eqref{embed-approx} for times $|t|> T$.

We next consider times $|t|\leq T$. We have
\begin{align}
e_n(t,x) & = -\mu[(\chi_n - \chi_n^{\alpha+1}) |w_n|^{\alpha} w_n](t,x-x_n) \label{embed-error6} \\
&\quad  + [w_n\Delta \chi_n + 2\nabla \chi_n\cdot\nabla w_n](t,x-x_n) \label{embed-error7} \\
&\quad - \tfrac{a}{|x|^2}[\chi_n w_n](t,x-x_n). \label{embed-error8}
\end{align}

First, by Sobolev embedding (noting that $r<d$) and \eqref{w-good}, 
\begin{align*}
\|\nabla\eqref{embed-error6}\|_{L_t^2 L_x^{\frac{2d}{d+2}}} & \lesssim \|w_n\|_{L_t^{q_1} L_x^{r_1}}^\alpha\bigl\{\|\nabla \chi_n\|_{L_x^d}  \|w_n\|_{L_t^q L_x^{\frac{dr}{d-r}}}+\|\nabla w_n\|_{L_t^q L_x^r}\bigr\} \\
& \lesssim \| |\nabla|^{s_c}w_n\|_{L_t^{q_1} L_x^{r_2}}^\alpha \|\nabla w_n\|_{L_t^q L_x^r} \lesssim 1,
\end{align*}
while by dominated convergence and \eqref{embed-conv} we get
\[
\|\eqref{embed-error6}\|_{L_{t,x}^{\frac{2(d+2)}{d+4}}} \lesssim \|w_n\|_{L_{t,x}^{q_0}}^\alpha \|(\chi_n^{\alpha+1}-\chi_n)w_n\|_{L_{t,x}^{\frac{2(d+2)}{d}} }\to 0
\]
as $n\to\infty$.  Thus, by interpolation,
\[
\| \eqref{embed-error6}\|_{L_t^\rho \dot H_a^{s_c,\gamma}} + \| \eqref{embed-error6}\|_{L_t^{\rho+}\dot H_a^{s,\gamma-}} \to 0\qtq{as}n\to\infty. 
\]

Next, using \eqref{embed-persist},
\begin{align*}
\|\nabla \eqref{embed-error7}\|_{L_t^1 L_x^2(\{|t|\leq T\}\times\R^d
)} & \lesssim T\bigl\{ \|\nabla\Delta \chi_n\|_{L_x^\infty}
    \|w_n\|_{L_t^\infty L_x^2} + \|\Delta\chi_n\|_{L_x^\infty}\|\nabla w_n\|_{L_t^\infty L_x^2}\\
&\quad\quad\quad + \|\nabla\chi_n\|_{L_x^\infty}
    \|\Delta w_n\|_{L_t^\infty L_x^2} \bigr\} \\
&\lesssim T\bigl\{ |x_n|^{-3}+|x_n|^{-2}+ |x_n|^{-1+\theta} \bigr\} \to 0 \qtq{as} n\to\infty.
\end{align*}
Similarly,
\begin{align*}
\|\eqref{embed-error7}\|_{L_t^1 L_x^2(\{|t|\leq T\}\times\R^d)} & \lesssim T\bigl\{ \|\Delta\chi_n\|_{L_x^\infty}
    \|w_n\|_{L_t^\infty L_x^2} + \|\nabla \chi_n \|_{L_x^\infty} \|\nabla w_n\|_{L_t^\infty L_x^2} \bigr\} \\
&\lesssim T\bigl\{ |x_n|^{-2} + |x_n|^{-1} \}\to 0 \qtq{as} n\to\infty.
\end{align*}
Thus
\[
\lim_{T\to\infty}\lim_{n\to\infty} \||\nabla|^{\theta} \eqref{embed-error7}\|_{L_t^1 L_x^2(\{|t|\leq T\}\times\R^d)} = 0\qtq{for all}\theta\in[0,1]. 
\]
Finally, we estimate
\begin{align*}
\|\langle\nabla\rangle \eqref{embed-error8} \|_{L_t^1 L_x^2(\{|t|\leq T\}\times\R^d)}
    & \lesssim T\bigl\{ \|\tfrac{\chi_n}{|\cdot + x_n|^2}\|_{L_x^\infty}\|\langle\nabla\rangle w_n\|_{L_t^\infty L_x^2} \\
& \quad \quad + \|\nabla\bigl(\tfrac{\chi_n}{|\cdot + x_n|^2}\bigr)\|_{L_x^\infty}\|w_n\|_{L_t^\infty L_x^2} \bigr\} \\
&\lesssim T\bigl\{ |x_n|^{-2} + |x_n|^{-3} \}\to 0\qtq{as} n\to\infty,
\end{align*}
so that
\[
\lim_{T\to\infty}\lim_{n\to\infty} \| |\nabla|^{\theta}\eqref{embed-error8}\|_{L_t^1 L_x^2(\{|t|\leq T\}\times\R^d)} = 0\qtq{for all}\theta\in[0,1]. 
\]
This completes the proof of \eqref{embed-approx} for times $|t|\leq T$.

With \eqref{embed-bounds}, \eqref{embed-data}, and
\eqref{embed-approx} in place, we apply Theorem~\ref{T:stab} to
deduce the existence of a global solution $v_n$ to \eqref{nls} with
$v_n(0) = \phi_n$ satisfying
\begin{align}
 \| v_n\|_{S_a^1(\R)} &\lesssim 1 \qtq{uniformly in} n,\nonumber \\
\label{embed-conclude-close}
 \lim_{T\to\infty} \limsup_{n\to\infty} \|[ v_n(\cdot-t_n) - \tilde v_{n,T}(\cdot)] \|_{\dot S_a^{\theta}(\R)} &= 0\qtq{for}\theta\in[s_c,s].
\end{align}

\textbf{4.} Finally, we turn to \eqref{embed-cc}.  We will only prove the approximation in the space $L_t^q \dot H_a^{s,r}$.  Approximation in the other spaces follows from similar arguments. 

Fix $\eps>0$.  As $C_c^\infty(\R\times\R^d)$ is dense in $L_t^q \dot H_x^{s,r}$, we may find $\psi_\eps\in C_c^\infty$ such that
\[
\|w_\infty - \psi_\eps\|_{L_t^q \dot H_x^{s,r}} < \tfrac{\eps}{3}.
\]
In light of \eqref{embed-conv} and \eqref{embed-conclude-close}, it suffices to show that
\begin{equation}\label{embed-cc-nts}
\| \tilde v_{n,T}(t,x) - w_\infty(t,x-x_n) \|_{L_t^q \dot H_x^{s,r}} < \tfrac{\eps}{3}
\end{equation}
for $n,T$ large.

Using \eqref{w-good} and equivalence of Sobolev spaces, we have 
\begin{equation}\label{embed-bd-more}
\|\tilde v_{n,T}\|_{L_t^q \dot H_a^{s+,r}} \lesssim 1\qtq{uniformly in}n,T.
\end{equation}
Using this and the triangle inequality we have
\[
\| \tilde v_{n,T}(t,x) - w_\infty(t,x-x_n)]\|_{L_t^q \dot H_a^{s+,r}} \lesssim 1.
\]

On the other hand, we can estimate
\begin{align*}
\|& \tilde v_{n,T}(t,x) - w_\infty(t,x-x_n)\|_{L_t^q L_x^r} \\
&\lesssim \|\chi_n w_n-w_\infty\|_{L_t^q L_x^r} \\
&\quad + \|e^{-i(t-T)\L^n}[\chi_n w_n(T)] - w_\infty\|_{L_t^q L_x^r((T,\infty)\times\R^d)} \\
& \quad + \|e^{-i(t+T)\L^n}[\chi_n w_n(-T)] - w_\infty \|_{L_t^q L_x^r((-\infty,-T)\times\R^d)}.
\end{align*}
The first term converges to zero as $n\to\infty$ by the dominated convergence theorem and \eqref{w-good}. The second and third terms are similar, so we only consider the second. For this term, we apply the triangle inequality. By \eqref{w-good} and monotone convergence,
\[
\|w_\infty\|_{L_t^q L_x^r((T,\infty)\times\R^d)}\to 0\qtq{as}T\to\infty,
\]
while arguing as we did for \eqref{embed-error1} we see that
\[
\lim_{T\to\infty}\limsup_{n\to\infty} \| e^{-it\L^n}[\chi_n w_n(T)]\|_{L_t^q L_x^r((0,\infty)\times\R^d)}=0.
\]
Interpolation now yields \eqref{embed-cc-nts} for $n,T$ large.  This completes the proof of Theorem~\ref{T:embedding}. \end{proof}
\section{Existence of minimal blowup solutions}\label{S:exist}

In this section we use the profile decomposition and stability theory to show the existence of minimal blowup solutions under the assumption that Theorem~\ref{T:main}(ii) or Theorem~\ref{T:main2} fails.  We first define
\[
L(\E) : = \sup\bigl\{\|u\|_{L_{t,x}^{q_0}(I\times\R^d)}\bigr\},
\]
where the supremum is taken over all maximal-lifespan solutions $u:I\times\R^d$ such that $M(u)^{\sigma}E_a(u) \leq \E$. In addition, for the focusing case we restrict to solutions satisfying
\begin{align*}
 \|u(t)\|^{\sigma}_{L_x^2} \|u(t)\|_{\dot H^1_a} < \K_a,
\end{align*}
for some $t\in I$.

By Theorem~\ref{T:LWP} and Corollary~\ref{R:coercive}, we have that $L(\E)<\infty$ for all $\E$ sufficiently small; in fact,\begin{equation}\label{small-data-bds}
L(\E) \lesssim \E^{\frac14}\qtq{for}  0<\E\lesssim \eta_0,
\end{equation}
where $\eta_0$ is the small-data threshold.

If Theorem~\ref{T:main}(ii) or Theorem~\ref{T:main2}   fails, then we see that there must exist a `critical' $\E_c\in (0, \E_{a})$ such that
\[
L(\E)<\infty\text{ for } \E<\E_c\qtq{and}  L(\E)=\infty\text{ for }\E>\E_c.
\]

\begin{theorem}[Existence of minimal blowup solutions]\label{T:exist} Suppose that either Theorem~\ref{T:main}(ii) or Theorem~\ref{T:main2} fails.  Then there exists a global solution $v$ to \eqref{nls} satisfying:
\[
M(v) = 1,\quad E_a(v) = \E_c, \qtq{and}
\|v\|_{L_{t,x}^{q_0}((-\infty,0)\times\R^d)} = \|v\|_{L_{t,x}^{q_0}((0,\infty)\times\R^d)} = \infty.
\]
Moreover, the orbit  of $\{v(t)\}_{t\in\R}$ is precompact in $H_x^1(\R^d)$.  Furthermore, in the focusing case, we have $\|v(0)\|_{\dot H_a^1} < \K_a$.
\end{theorem}

\begin{proposition}[Palais--Smale condition]\label{P:PS} Let $(d,a,\alpha)$ satisfy \eqref{important}.  Let $u_n:I_n\times\R^d\to\C$ be a sequence of solutions to \eqref{nls} such that $M(u_n)E_a(u_n) \nearrow \E_c$, and suppose $t_n\in I_n$ satisfy
\begin{align}\label{PS-belowQ'}
\lim_{n\to\infty} \|u_n\|_{L_{t,x}^{q_0}(\{t< t_n\}\times\R^d)}=\lim_{n\to\infty} \|u_n\|_{L_{t,x}^{q_0}(\{t> t_n\}\times\R^d)}=\infty
\end{align}
in the defocusing case, or
\begin{equation}
\label{PS-belowQ}
\begin{aligned}\left\{
  \begin{array}{ll}
  \|u_n(t_n)\|^{\sigma}_{L_x^2} \|u_n(t_n)\|_{\dot H^1_a} < \K_a,  \\
   \displaystyle\lim_{n\to\infty} \|u_n\|_{L_{t,x}^{q_0}(\{t< t_n\}\times\R^d)}=\displaystyle\lim_{n\to\infty} \|u_n\|_{L_{t,x}^{q_0}(\{t> t_n\}\times\R^d)}=\infty
  \end{array}\right.
\end{aligned}
\end{equation}
in the focusing case.

Then, with $\lambda_n := M(u_n)$, we have that $\{u_n^{\lambda_n}(t_n)\}$ converges along a subsequence in $H_x^1$, where we use the notation from \eqref{scaling}.
\end{proposition}

Given Proposition~\ref{P:PS}, it is standard to complete the proof of Theorem~\ref{T:exist} (cf. \cite{KVMZ}, for example). Thus, it remains to prove Proposition~\ref{P:PS}.

\begin{proof}[Proof of Proposition~\ref{P:PS}] Without loss of generality, we assume $\lambda_n\equiv 1$ (equivalently, $M(u_n)\equiv 1$).  We will give the proof in the focusing case; the defocusing case is essentially the same, with a few simplifications. 

We have that each $u_{n}$ is global by Corollary~\ref{R:coercive}. By time-translation invariance, we may assume $t_n\equiv 0$; thus, we have
\begin{equation}\label{PS-blowup}
\lim_{n\to\infty} \| u_n\|_{L_{t,x}^{q_0}((0,\infty)\times\R^d)} = \lim_{n\to\infty} \| u_n\|_{L_{t,x}^{q_0}((-\infty,0)\times\R^d)} = \infty.
\end{equation}
Note that $E_a( u_n)\to \E_c$, and \eqref{PS-belowQ'} or \eqref{PS-belowQ} holds for $ u_n$.

Now we apply Proposition~\ref{P:LPD} to $\{ u_n(0)\}$ to get the decomposition
\[
 u_n(0) = \sum_{j=1}^J \phi_n^j + r_n^J \qtq{for all finite}  0 \leq J \leq J^*\in\{0,1,2,\dots,\infty\},
\]
which satisfies the conclusions of Proposition~\ref{P:LPD}.  We need to show that $J^*=1$, $r_n^1\to 0$ in $H^1$, $t_n^1\equiv 0$, and $x_n^1\equiv 0$. 

We first claim that $\liminf_n E_a(\phi_n^j)>0$ for each $j$.  To this end, we first note that if $|x_n^j|\to\infty$, then \eqref{coo4} gives 
\begin{equation}\label{onebub-conv}
\|\phi_n^j\|_{H_a^1} \to \|\phi^j\|_{H_x^1}>0.
\end{equation}
Thus the claim follows from \eqref{decouple1}, \eqref{PS-belowQ}, and Proposition~\ref{P:coercive}a.(iii).

We are left with two possibilities: either (a) $\sup_j\limsup_{n\to\infty} E_a(\phi_n^j)=\E_c$, or (b) $\sup_j \limsup_{n\to\infty} E_a(\phi_n^j)<\E_c-3\delta$ for some $\delta>0$.  We will show that in the scenario (a) we have the desired compactness, while scenario (b) cannot happen.

\textbf{Scenario (a).} In this case, we have $J^*=1$ and $u_n(0)=\phi_n + r_n$ with $r_n\to 0$ in $\dot H^1$. (We consider the issue of $L^2$ convergence below.)

We first show that we must have $x_n\equiv 0$. If not, then we will apply Theorem~\ref{T:embedding}.  This requires that we check \eqref{embedding-threshold}, which is clear in this scenario if $t_n\equiv 0$ and requires Corollary~\ref{L4-to-zero} if $t_n\to\pm\infty$ (cf. Corollary~\ref{C:thresholds}).  In particular Theorem~\ref{T:embedding} gives a global solution $v_n$ to \eqref{nls} with $v_n(0)=\phi_n$ obeying global space-time bounds. However, an application of Theorem~\ref{T:stab} then yields uniform space-time bounds for the $u_n$, which is a contradiction to \eqref{PS-blowup}.

We next show that we must have $t_n\equiv 0$.  If $t_n\to\infty$, say, then an application of Theorem~\ref{T:stab} (comparing $u_n$ to the linear solutions $e^{-it\L}u_n(0)$) suffices to give uniform space-time bounds for the $u_n$, resulting in a contradiction to \eqref{PS-blowup}.  The assumptions $t_n\to\infty$ and the condition \eqref{rnJ} guarantee that the linear solutions are actually approximate solutions. 

To complete the proof in scenario (a), we need to show that $r_n\to 0$ in $L^2$. As $M(u_n)\equiv 1$, it suffices to show $\|\phi\|_{L^2}=1$.  If not, then $\|\phi\|_{L^2}<1$ and hence (by definition of $\E_c$) the solution to \eqref{nls} with data $\phi$ would scatter.  Using the fact that $r_n\to 0$ in $\dot H^{s_c}$ (since it is bounded in $L^2$ and converges to zero in $\dot H^1$), another application of stability theory would imply space-time bounds for the $u_n$, giving a contradiction to \eqref{PS-blowup}.  This completes the proof in scenario (a).

\textbf{Scenario (b).} In this case, we will find a contradiction.  Note that for every finite $J\leq J^*$, we have
\[
M(\phi_n^j)^{\sigma} E_a(\phi_n^j)< \E_c-2\delta\qtq{for} 1\leq j\leq J\qtq{and} n\text{ large.}
\]
Using \eqref{decouple1}, \eqref{PS-belowQ}, and Proposition~\ref{P:coercive}, we also have that for some $\delta' >0,$
\begin{align}\label{name2}
\|\phi_n^j\|^{\sigma}_{L_x^2} \|\phi_n^j\|_{\dot H_a^1} < (1-\delta')\K_a\qtq{for all}1 \leq j\leq J\qtq{and}  n\text{ large.}
\end{align}

If $|x_n^j|\to\infty$, then we argue as above to get a global solution $v_n^j$ to \eqref{nls} with $v_n^j(0)=\phi_n^j$.  If $x_n^j\equiv 0$ and $t_n^j\equiv 0$ then we take $v^j$ to be the solution to \eqref{nls} with $v^j(0)=\phi^j$.  If $x_n^j\equiv 0$ and $t_n^j\to\pm\infty$ then we use Theorem~\ref{T:LWP} to find a solution $v^j$ to \eqref{nls} that scatters to $e^{-it\L}\phi^j$ in $H^1$ as $t\to\infty$. In the latter two cases we define
\[
v_n^j(t,x) = v^j(t+t_n^j, x). 
\]

Note that
\[
\lim_{n\to\infty} \|v_n^j(0) - \phi_n^j\|_{H_a^1} = 0
\]
Thus $E_a(v_n^j) \leq \E_c - \delta$ for all $1\leq j\leq J$ and $n$ large.  By definition of $\E_c$ and \eqref{name2}, it follows that each $v_n^j$ is global in time with uniform space-time bounds.  In particular (using Theorem~\ref{T:embedding} for the $j$ for which $|x_n^j|\to\infty$), for any $\eta>0$ we may find $\psi_\eta^j\in C_c^\infty(\R\times\R^d)$ such
\begin{equation}\label{PS-cc}
\|v_n^j-\psi_\eta^j(\cdot-t_n^j,\cdot - x_n^j)\|_{X(\R\times\R^d)}  < \eta
\end{equation}
for all $n$ sufficiently large, where $X$ is any of the norms appearing in \eqref{Xin}. 

We will now apply Theorem~\ref{T:stab} to get a contradiction to  \eqref{PS-blowup}. We define
\begin{align}\nonumber
&u_n^J(t) : = \sum_{j=1}^J v_n^j(t) + e^{-it\L} r_n^J,\qtq{which satisfies}\\
\label{L:approx-data}
&\lim_{n\to\infty}\|u_n^J(0) -  u_n(0) \|_{H_x^1} = 0 \qtq{for any}  J.
\end{align}
 In order to apply Theorem~\ref{T:stab}, we need to verify the following conditions:
\begin{align}
\label{PS-approx1}
& \limsup_{n\to\infty} \bigl\{\|u_n^J(0)\|_{H_x^1} + \|u_n^J\|_{L_{t,x}^{q_0}} \bigr\}\lesssim  1\qtq{uniformly in}  J, \\
\label{PS-approx2}
& \lim_{J\to J^*}\limsup_{n\to\infty} \||\nabla|^{s_c}\bigl[(i\partial_t -\L)u_n^J -\mu |u_n^J|^{\alpha} u_n^J\bigr]\|_{N(\R)} = 0,
\end{align}
where the space-time norms are taken over $\R\times\R^d$. (See Theorem~\ref{T:stab} for the definition of $N(\R)$).  Assuming that \eqref{PS-approx1} and \eqref{PS-approx2} hold, Theorem~\ref{T:stab} implies that the $u_n$ inherit the space-time bounds from the $u_n^J$, which contradicts \eqref{PS-blowup}.

Thus, to complete the proof of Proposition~\ref{P:PS}, it remains to prove \eqref{PS-approx1} and \eqref{PS-approx2}.  First, we record some important orthogonality conditions.  We recall the spaces appearing in \eqref{Xin}.   Then for $j\neq k$, we have
\begin{align*}
\|v_n^j& v_n^k\|_{L_{t,x}^{\frac{q_0}{2}}\cap L_{t,x}^{\frac{d+2}{d}}} + \|(\L^{\frac{s_c}{2}}v_n^j )(\L^{\frac{s_c}{2}} v_n^k)\|_{L_t^{\frac{q_1}{2}} L_x^{\frac{r_1}{2}}} \\
& \quad + \|(\L^{\frac{s}{2}}v_n^j)(\L^{\frac{s}{2}}v_n^k)\|_{L_t^{\frac{q}{2}} L_x^{\frac{r}{2}}} \to 0 \qtq{as}n\to\infty. 
\end{align*}
Indeed, we knew that each $v_n^j(t,x)$ were of the form $\psi^j(t-t_n^j,x-x_n^j)$ for some $\psi^j\in C_c^\infty(\R\times\R^d)$, this would follow directly from a change of variables and \eqref{orthogonal}. In fact, \eqref{PS-cc} tells us that we may estimate each $v_n^j$ by such a function (in suitable spaces) up to arbitrarily small errors. Using this together with the uniform bounds on the $v_n^j$, the result follows.

\begin{proof}[Proof of \eqref{PS-approx1}] First, using \eqref{PS-belowQ} and \eqref{L:approx-data}, we deduce the $H_x^1$-bound in \eqref{PS-approx1}.  From this bound and the decoupling \eqref{decouple1}, we deduce that
\[
\limsup_{n\to\infty}\sum_{j=1}^J \|\phi_n^j\|_{H_x^1}^2 \lesssim 1\qtq{uniformly in } J.
\]
In fact, in view of \eqref{iso}, \eqref{onebub-conv}, and the definition of profiles, this implies
\[
\sum_{j=1}^\infty \|\phi^j\|_{H_x^1}^2 < \infty
\]

Letting $\eta_0>0$ be the small-data threshold of Theorem~\ref{T:LWP} and using \eqref{small-data-bds}, there exists $J_0=J_0(\eta_0)$ such that
\[
\sup_J \limsup_{n\to\infty} \sum_{j=J_0}^J \|v_n^j\|_{S_a^1(\R)}^2 \lesssim \limsup_{n\to\infty}\sum_{j\geq J_0} \|\phi^j\|_{H^1}^2 <\eta_0.
\]
Thus, we deduce that
\begin{equation}\label{PS-summable}
\limsup_{n\to\infty} \sum_{j=1}^J \| v_n^j\|_{S_a^1(\R)}^2 \lesssim 1\qtq{uniformly in} J.
\end{equation}

Next, we recall (cf. \cite{Keraani}, for example) that for $p>1$, 
\begin{align*}
  \biggl|\bigl|\sum_{j=1}^{J}z_{j}\bigr|^{p}-\sum_{j=1}^{J}|z_{j}|^{p}\biggl|\lesssim_J \sum_{j\neq k}|z_{j}||z_{k}|^{p-1}.
\end{align*}Using orthogonality, Sobolev embedding, and equivalence of Sobolev spaces,
\begin{align*}
\biggl|\ \biggl\|\sum_{j=1}^J v_n^j\biggr\|_{L_{t,x}^{q_0}}^{q_0} - \sum_{j=1}^J \|v_n^j\|_{L_{t,x}^{q_0}}^{q_0}\biggr| &
\lesssim_J \sum_{j\neq k} \|v_n^j\|_{L_{t,x}^{q_0}}^{q_0-2} \|v_n^j v_n^k\|_{L_{t,x}^{\frac{q_0}{2}}}\to 0\qtq{as} n\to\infty.
\end{align*}
As $\|e^{-it\L}r_n^J\|_{L_{t,x}^{q_0}}\lesssim 1$ uniformly, we may therefore deduce  the $L_{t,x}^{q_0}$ bound in \eqref{PS-approx1} from \eqref{PS-summable}.  This completes the proof of \eqref{PS-approx1}.\end{proof}

Before turning to \eqref{PS-approx2}, we collect a few more bounds for the $u_n^J$.  In particular, we claim
\begin{equation}\label{unJ-sym}
\limsup_{n\to\infty} \|u_n^J\|_{L_{t,x}^{\frac{2(d+2)}{d}}\cap L_t^{q_1} \dot H_a^{s_c,r_1}\cap L_t^q \dot H_a^{s,r}} \lesssim 1 \qtq{uniformly in}J.
\end{equation}
Indeed, we can argue as above for the first norm.  For the second norm, we argue as follows:
\begin{align*}
\biggl\| \sum_{j=1}^J v_n^j\bigg\|_{L_t^{q_1} \dot H_a^{s_c,r_1}}^2 & \lesssim \biggl\| \biggl(\sum_{j=1}^J \L^{\frac{s_c}{2}} v_n^j\biggr)^2\biggr\|_{L_t^{\frac{q_1}{2}}L_x^{\frac{r_1}{2}}} \\
& \lesssim \sum_{j=1}^J \|v_n^j\|_{L_t^{q_1}\dot H_a^{s_c,r_1}}^2 + C_J\sum_{j\neq k} \|(\L^{\frac{s_c}{2}}v_n^j)(\L^{\frac{s_c}{2}}v_n^k)\|_{L_t^{\frac{q_1}{2}}L_x^{\frac{r_1}{2}}}. 
\end{align*}
Thus, by \eqref{PS-summable} and orthogonality, we deduce the required bound for $u_n^J$.  A similar argument treats the third norm in \eqref{unJ-sym}. 

\begin{proof}[Proof of \eqref{PS-approx2}] Denoting $F(z) = -\mu |z|^{\alpha} z$, we write
\begin{align}\label{enj1}
e_n^J:=(i\partial_t-\L)u_n^J -F(u_n^J) & =  \sum_{j=1}^J F(v_n^j) - F\bigl(\sum_{j=1}^J v_n^j\bigr) \\
\label{enj2}
& \quad + F(u_n^J-e^{-it\L}r_n^J) - F(u_n^J).
\end{align}

We first estimate \eqref{enj1}.  We recall another pointwise estimate from \cite{Keraani}, namely
\[
\biggl|\nabla\biggl(F\bigl(\sum_{j=1}^J v_n^j\bigr) - \sum_{j=1}^J F(v_n^j)\biggr)\biggr| \lesssim_J \sum_{j\neq k}|\nabla v_n^j|\,|v_n^k|^\alpha. 
\]
Thus we may estimate as in \eqref{H1se} to get
\begin{align*}
\limsup_{n\to\infty}\| \eqref{enj1}\|_{L_t^2 \dot H_a^{1,\frac{2d}{d+2}}} & \lesssim_J \limsup_{n\to\infty}\sum_{j\neq k} \|v_n^k\|_{L_t^{q_1} \dot H_a^{s_c,r_2}}\| v_n^j\|_{L_t^q H_a^{1,r}}\lesssim_J 1. 
\end{align*}
On the other hand, using orthogonality,
\[
\limsup_{n\to\infty}\|\eqref{enj1}\|_{L_{t,x}^{\frac{2(d+2)}{d+4}}} \lesssim_J \limsup_{n\to\infty}\sum_{j\neq k}\|v_n^j v_n^k\|_{L_{t,x}^{\frac{q_0}{2}}} \|v_n^j\|_{L_{t,x}^{\frac{2(d+2)}{d}}} =0
\]
for each $J$.  By interpolation, this implies
\[
\limsup_{J\to J^*}\limsup_{n\to\infty} \| \eqref{enj1}\|_{L_t^\rho \dot H_a^{s_c,\gamma}} = 0. 
\]

We next estimate \eqref{enj2}.  We first have (estimating as in Remark~\ref{remark1} and recalling \eqref{unJ-sym})
\begin{align*}
\limsup_{n\to\infty} \| F(u_n^J)\|_{L_t^2 \dot H_a^{s,\frac{2d}{d+2}}} &\lesssim \limsup_{n\to\infty} \|u_n^J\|_{L_t^{q_1}\dot H_a^{s_c,r_1}}^\alpha \|u_n^J\|_{L_t^q \dot H_a^{s,r}} \\
& \lesssim 1 \qtq{uniformly in}J. 
\end{align*}
As $e^{-it\L}r_n^J\in S_a^1$, this estimate suffices to show
\[
\limsup_{n\to\infty} \|\eqref{enj2}\|_{L_t^2 \dot H_a^{s,\frac{2d}{d+2}}}\lesssim 1\qtq{uniformly in}J. 
\]
On the other hand, by Strichartz, \eqref{unJ-sym}, \eqref{PS-approx1}, and \eqref{rnJ}, we can bound
\begin{align*}
\limsup_{J\to J^*}&\limsup_{n\to\infty} \|\eqref{enj2}\|_{L_{t,x}^{\frac{2(d+2)}{d+4}}} \\
& \lesssim \limsup_{J\to J^*}\limsup_{n\to\infty} \bigl\{\|e^{-it\L}r_n^J\|_{L_{t,x}^{q_0}} \bigl(\|u_n^J\|_{L_{t,x}^{q_0}}+\|r_n^J\|_{\dot H^{s_c}}\bigr)\\
&\quad\quad\quad\quad\quad\quad\quad \quad\times \bigl(\|u_n^J\|_{L_{t,x}^{\frac{2(d+2)}{d}}}\!+\|r_n^J\|_{L_x^2}\bigr)\bigr\} = 0.
\end{align*}
Thus, by interpolation, 
\[
\limsup_{J\to J^*}\limsup_{n\to\infty} \|\eqref{enj2}\|_{L_t^{2-}\dot H_a^{s_c,\frac{2d}{d+2}+}} = 0.
\]
We conclude that \eqref{PS-approx2} holds.\end{proof}
As described above, \eqref{PS-approx1} and \eqref{PS-approx2} complete the proof of Proposition~\ref{P:PS}. \end{proof}

\section{Preclusion of minimal blowup solutions}\label{S:not-exist}

In this section, we prove that the solutions constructed in Section \ref{S:exist} cannot exist.  This ensures that $\E_{c}=\E_{a}$ if $\mu=-1$ and $\E_{c}=\infty$ if $\mu=1$.  In particular, this completes the proof of Theorem~\ref{T:main}(ii) and Theorem \ref{T:main2}.

\begin{theorem}\label{T:not-exist} There are no solutions to \eqref{nls} as in Theorem~\ref{T:exist}.
\end{theorem}

\begin{proof} Suppose towards a contradiction that there exists a global solution $v$ as in Theorem~\ref{T:exist}.  Let $\delta>0$ such that $\E_c\leq (1-\delta)\E_a$, and take $\eta>0$ to be determined later. By pre-compactness in $H_x^1$, there exists $R=R(\eta)>1$ such that
\begin{equation}\label{E:tight}
\int_{|x|>R} |u(t,x)|^2 + |\nabla u(t,x)|^2 + |u(t,x)|^{\alpha+2} \,dx < \eta \qtq{uniformly for} t\in\R.
\end{equation}

In the focusing case $(\mu=-1$), we use Proposition~\ref{P:coercive}a.(ii), which implies
\begin{equation}\label{E:lb}
\|u(t)\|_{\dot H_a^1}^2 - \tfrac{\alpha d}{2(\alpha+2)} \|u(t)\|_{L_x^{\alpha+2}}^{\alpha+2} \geq c \|u(t)\|_{\dot H_a^1}^2\gtrsim_u c\qtq{uniformly for} t\in\R
\end{equation}
for some $c=c(\delta,a)>0$.

In the defocusing case $(\mu=1$), we note that we have a uniform lower bound on the $\dot H_a^1$ norm of $u(t)$ by compactness and the fact that the solution is not identically zero (indeed, it has infinite $L_{t,x}^{q_0}$-norm). 

We now define $w_R$ as in Section~\ref{S:virial}.  Choosing $\eta=\eta(u,c)$ sufficiently small, we claim that 
\begin{align}\label{no name2}
c \lesssim_u \partial_{tt}\int_{\R^d} w_R(x) |u(t,x)|^2 \,dx \qtq{uniformly for}  t\in\R.
\end{align}
Indeed, this follows from Lemma~\ref{L:virial0} and \eqref{E:tight}, along with \eqref{E:lb} in the focusing case and the $\dot H^1$ lower bound in the defocusing case.

Using \eqref{virial} and noting that
\[
\biggl|\partial_t \int_{\R^d} w_R(x) |u(t,x)|^2\,dx\biggr| \lesssim R\|u\|_{L_t^\infty H_x^1}^2 \lesssim_u R \qtq{uniformly for} t\in \R,
\]
we can integrate \eqref{no name2} over any interval of the form $[0,T]$ and use the fundamental theorem of calculus to deduce that $cT \lesssim_u R$. Choosing $T$ sufficiently large now yields a contradiction.\end{proof}

\section{Blowup}\label{S:blowup}
%
In this section, we prove the blowup result Theorem~\ref{T:main}(i). The reader can refer to \cite{DHR, Glassey, HR, OT} for similar arguments; we give a complete proof for convenience.

\begin{proof}[Proof of Theorem~\ref{T:main}(i)] We let $u_0$ and $u$ be as in the statement of Theorem~\ref{T:main}(i).  We choose $\delta>0$ so that $M(u_0)E_a(u_0)\leq (1-\delta)\E_a$.

First, suppose $xu_0\in L^2$.  Using Lemma~\ref{L:virial0} and Proposition~\ref{P:coercive}, we deduce
\[
\partial_{tt} \int_{\R^d} |x|^2 |u(t,x)|^2  \,dx \leq -c < 0 \qtq{for all} t\in I
\]
for some $c=c(\delta, a, \|u_0\|_{L_x^2})$.  By the standard convexity arguments (cf. \cite{Glassey}), it follows that $u$ blows up in finite time in both time directions.

Next, suppose that $u_0$ is radial. By  H\"older's inequality, radial Sobolev embedding, and the equivalence of Sobolev spaces, the following holds: for any radial $f\in H_a^1$ and any $R>1$,
\[
\|f\|_{L_x^{\alpha+2}(\{|x|>R\})}^{\alpha+2} \lesssim R^{-\alpha} \|f\|_{L_x^2}^{\frac{\alpha+4}{2}} \| f\|^{\frac{\alpha}{2}}_{\dot H_a^1}.
\]
Now take $R>1$ to be determined below and define $w_R\geq 0$ as in Section~\ref{S:virial}.  Using Lemma~\ref{L:virial0} and the conservation of mass, we can bound
\begin{align*}
&\partial_{tt}\int_{\R^d} w_R(x) |u(t,x)|^2\,dx \leq 8\bigl[\|u(t)\|_{\dot H_a^1}^2 - \tfrac{\alpha d}{2(\alpha+2)}\|u(t)\|_{L_x^{\alpha+2}}^{\alpha+2}\bigr] + e(t), \quad\text{where}\\
&|e(t)| \lesssim R^{-2} \|u_0\|_{L_x^2}^2 + \|u(t)\|_{L_x^{\alpha+2}(\{|x|\geq R\})}^{\alpha+2}.
\end{align*}
Take $\eps=\eps(\delta)>0$ and $c=c(\delta,a,\|u_0\|_{L_x^2})>0$ as in Proposition~\ref{P:coercive}b.(iii).  By the radial Gagliardo--Nirenberg inequality, conservation of mass, and Young's inequality, we may bound
\begin{align*}
\|u(t)\|_{L_x^{\alpha+2}(\{|x|>R\})}^{\alpha+2}& \lesssim R^{-\alpha} \|u(t)\|_{L_x^2}^{\frac{\alpha+4}{2}} \| u(t)\|^{\frac{\alpha}{2}}_{\dot H_a^1}\\
&\leq 8\eps \|u(t)\|_{\dot H_a^1}^2 + C\eps^{-1} R^{-\frac{4\alpha}{4-\alpha}} \|u_0\|_{L_x^2}^{\frac{2(\alpha+4)}{4-\alpha}}\qtq{for some} C>0.
\end{align*}
Thus, using Proposition~\ref{P:coercive} and choosing $R=R(c, \eps, \|u_0\|_{L_x^2})$ sufficiently large, we can guarantee that
\[
\partial_{tt} \int_{\R^d} w_R(x) |u(t,x)|^2\,dx \leq - \tfrac c2 <0,
\]
which again implies that $u$ must blow up in finite time in both time directions.\end{proof}


\end{document}